\documentclass[11pt,a4paper]{article}

\usepackage{lmodern}
\usepackage[T1]{fontenc}
\usepackage{hyperref}
\usepackage{amsfonts,amsmath,amssymb,amsopn}
\usepackage[all]{xy}
\usepackage{vmargin}
\usepackage{makeidx}
\usepackage{graphicx}
\usepackage{tikz}
\usepackage{comment}
\usepackage{amsthm}
\title{Twisted calculus on affinoid algebras}
\author{Bernard Le Stum \& Adolfo Quir\'os\thanks{Supported by grant MTM2015-68524-P (MINECO/FEDER, Spain/EU).}}
\date{Version of \today}
\usepackage[style=alphabetic,citestyle=alphabetic,sorting=nyt,sortcites=true,autopunct=true,babel=hyphen,hyperref=true,abbreviate=false,backref=false]{biblatex}
\AtEveryBibitem{ \clearfield{url} \clearfield{urldate} \clearfield{review} \clearfield{series} \clearfield{doi} \clearfield{isbn} \clearfield{issn} } 
\defbibheading{bibempty}{}
\DeclareNolabel{\nolabel{\regexp{[\p{Z}]+}}}
\newtheorem{thm}{Theorem}[section]
\newtheorem{prop}[thm] {Proposition}
\newtheorem{cor}[thm] {Corollary}
\newtheorem{lem}[thm] {Lemma}
\newtheorem{dfn}[thm] {Definition}

\newtheorem*{thm*}{Theorem}

\newenvironment{xmp}[1][Example]{\begin{trivlist} \item[\hskip \labelsep {\bfseries #1}]}{\end{trivlist}}
\newenvironment{xmps}[1][Examples]{\begin{trivlist} \item[\hskip \labelsep {\bfseries #1}]}{\end{trivlist}}
\newenvironment{rmk}[1][Remark]{\begin{trivlist} \item[\hskip \labelsep {\bfseries #1}]}{\end{trivlist}}
\newenvironment{rmks}[1][Remarks]{\begin{trivlist} \item[\hskip \labelsep {\bfseries #1}]}{\end{trivlist}}
\begin{document}

\maketitle

\bigskip

\begin{center}
\textbf{Abstract}
\end{center}

We introduce the notion of a twisted differential operator of given radius relative to an endomorphism $\sigma$ of an affinoid algebra $A$. We show that this notion is essentially independent of the choice of the endomorphism $\sigma$. As a particular case, we obtain an explicit equivalence between modules endowed with a usual integrable connection (i.e. differential systems) and modules endowed with a $\sigma$-connection of the same radius (i.e. $q$-difference systems). Moreover, this equivalence preserves cohomology and in particular solutions.

\medskip

\tableofcontents

\section*{Introduction}
\addcontentsline{toc}{section}{Introduction}

Confluence is the process that consists in replacing a differential equation with a difference or $q$-difference equation in order to get a good approximation of the solutions. In ultrametric analysis, it happens that one can do a lot better:  with some mild extra hypothesis, there exists a difference or $q$-difference equation that has exactly the \emph{same} solutions. Actually, there exists a one to one correspondence between differential equations on one side and difference or $q$-difference equations on the other (that induces a bijection on the solutions).

The last years have seen a lot of progress in this direction.
In \cite{AndreDiVizio04}, Yves Andr\'e and Lucia Di Vizio compute the tannakian group of $q$-difference equations (with Frobenius) over the Robba ring. It happens to be the exactly the same as the tannakian group of differential equations over the Robba ring.
 As a corollary, they obtain an equivalence of categories.
Shortly after, Andrea Pulita in \cite{Pulita08} could prove a general equivalence between differential equations and difference equations on an affinoid open subset of the line:
the main idea is to use formal solutions as a bridge between the two worlds. In \cite{Pulita17}, he improves on his result in various directions.

Our approach consists in generalizing the notion of differential operator so that it includes difference operators and $q$-difference operators as we did in the algebraic situation in \cite{LeStumQuiros18}.
In fact, this is very general.
First of all, if $R$ is any affinoid algebra, $A$ is an affinoid $R$-algebra and $\sigma$ is an endomorphism of $A$, one introduces the notion of \emph{$\sigma$-coordinate} $x$ on $A$ that generalizes the notion of \'etale coordinate $x$ when $A$ is smooth over $R$ and $\sigma = \mathrm{Id}_{A}$.
The cases $\sigma(x) = x + h$ and $\sigma(x) = qx$ will provide difference and $q$-difference operators respectively (and they can be unified by using $\sigma(x) = qx + h$).
With some conditions on $A$ and $\sigma$, one can define the ring $D_{A/R,\sigma}^{(\eta)}$ of $\sigma$-differential operators of radius $\eta$ on $A$.
We show in theorem \ref{rnghom} that if $\tau$ is any other such $R$-endomorphism of $A$, then there exists a canonical isomorphism
\[
D_{A/R,\sigma}^{(\eta)} \simeq D_{A/R,\tau}^{(\eta)}.
\]
This is totally explicit in the sense that we have for example
\[
\partial_{\sigma} = \sum_{k=1}^\infty \left( \prod_{i=1}^{k-1} \left( \sigma(x) - \tau^i(x)\right)\right) \partial_{\tau}^{[k]}.
\]
From the particular cases $\tau = \mathrm{Id}_{A}$ and $\sigma(x) = qx + h$, we deduce our confluence theorem \ref{confl} (see also theorem 6.3 ii) of \cite{Pulita08}).
It would be necessary to introduce some more vocabulary in order to state now this theorem in full generality.
Nevertheless, as an illustration, we can indicate the following corollary (that makes the link with the original result of Andr\'e and Di Vizio as well as with the work of Pulita):  

\begin{thm*}
Let $K$ be a non trivial complete ultrametric field of characteristic zero and $X$ the closed annulus $r_{1} \leq |x| \leq r$.
Let $q,h \in K$ and $\eta \geq 0$ be such that
\[
\|h\| \leq \eta < r_{1} , \quad \|1 - q\| \leq \frac \eta r,
\]
and $q$ is \emph{not} a root of unity.
Then, $X$ is globally stable under the endomorphism $\sigma(x) = qx + h$, and there exists a fully faithful functor
\[
\nabla \mathrm{-Mod}^{(\eta\dagger)}(X) \hookrightarrow \sigma\mathrm{-Mod}(X)
\]
from the category of modules with connection whose (naive) radius of convergence is at least $\eta$, to the category of $\sigma$-modules.
\end{thm*}

We want to emphasize again the fact that, even if this functor already exists in Pulita's work, our methods provide explicit formulas: we will have
\[
\sigma(s) = \sum_{k=0}^\infty  \frac {((q-1)x+h)^{k}}{k!} \partial^{k}(s).
\]
Moreover, we obtain an isomorphism on cohomology $H^*_{\partial}(M) = H^*_{\sigma}(M)$.

Note that we exclude here the case where $q$ is a root of unity.
Actually, there exists also a formal confluence theorem (theorem 9.13 of \cite{LeStumQuiros18}) in this situation but it is a lot more technical.
This is however very interesting because, as the first author showed with Michel Gros in \cite{GrosLeStum13} (see also the more recent \cite{GrosLeStumQuiros17}), there exists a quantum Simpson correspondence when $q$ is a root of unity.
One can hope that an ultrametric version of these theorems could provide an ultrametric Simpson correspondence.
One could also use the analytic density lemma \eqref{andens} below and the Azumaya nature of the ring of quantum differential operators in order to attack Dixmier's conjecture (Problem 11.1 of \cite{Dixmier68}).
We intend to come back on these questions in the future.

We fix a non trivial complete ultrametric field $K$ and refer the reader to \cite{BoschGuntzerRemmert84} for the theory of affinoid algebras.
When we consider a real number, we always assume that some power of it lies in $|K|$.
We recall that an \emph{affinoid algebra} is a quotient $A$  (by an ideal) of a \emph{Tate algebra}
\[
K\{T_{1}, \ldots, T_{N}\} := \left\{\sum_{i\geq 0} a_{\underline i} T^{\underline i}, \quad |a_{i} | \to 0 \right\}.
\]
The quotient norm (of the Gauss norm) turns $A$ into a Banach algebra and any morphism of affinoid algebras is automatically continuous (and any ideal is closed) for any Banach structures.
We recall that a \emph{finite $A$-module} is a quotient $M$ (by a submodule) of a free module $A^r$.
The quotient norm turns it into a Banach $A$-module and any morphism between finite $A$-modules is continuous (and any submodule is closed) for any Banach structure.
Note however that tensor product of affinoid algebras has to be replaced by \emph{completed tensor product} in order to stay inside the category of affinoid algebras.

We also want to recall from definition 2.5 of \cite{LeStumQuiros15}, that the \emph{quantum binomial coefficients} are defined by the \emph{quantum Pascal identities} (in a ring $R$ for some $q \in R$)
\[
{n \choose k}_{q} = {n-1 \choose k-1}_{q} + q^k {n-1 \choose k}_{q}
\]
for $n,k \geq 1$ with initial conditions
\[
{n \choose 0}_{q} = 1 \quad \mathrm{and} \quad {0 \choose k}_{q} = 0 \quad \mathrm{for} \quad k \neq 0.
\]
We will also consider the \emph{quantum integers} $(n)_{q} := {n \choose 1}_{q}$ and the \emph{quantum factorials} $(n)_{q}! := (n)_{q} (n-1)_{q} \cdots (2)_{q}(1)_{q}$.
We will use the prefix ``$q$-'' instead of the attribute \emph{quantum} when we want to specify the parameter.

Throughout the article, $R$ denotes an affinoid algebra (implicitly endowed with a fixed Banach norm).

\section{Twisted calculus}

We explain here how some material introduced in \cite{LeStumQuiros18b} and \cite{LeStumQuiros18} has to be modified (or not) in order to work in the setting of affinoid algebras.

\subsection{Twisted affinoid algebras}

We define a \emph{twisted affinoid $R$-algebra} as an affinoid $R$-algebra $A$ endowed with an endomorphism $\sigma$.
They form a category in the obvious way.
A Banach norm on a twisted affinoid algebra will be called \emph{contractive} if $\sigma$ is contractive for this norm.
For example, when $A$ is reduced, then the spectral norm is contractive no matter what $\sigma$ is.

\begin{lem}
If $A$ is a twisted affinoid $R$-algebra and $\pi : R\{T_{1}, \ldots, T_{N}\} \twoheadrightarrow A$ is a surjective homomorphism, then the following are equivalent:
\begin{enumerate}
\item the quotient norm is contractive,
\item $\sigma$ lifts to $R\{T_{1}, \ldots, T_{N}\}$.
\end{enumerate}
\end{lem}

\begin{proof}
The second condition means that there exists for all $i = 1, \ldots, N$, $F_{i} \in R\{T_{1}, \ldots, T_{N}\}$ such that $\|F_{i}\| \leq 1$ and $(\sigma \circ \pi)(T_{i}) = \pi(F_{i})$.
This is equivalent to saying that, for the quotient norm, we have for all $i = 1, \ldots, N$, $\|(\sigma \circ \pi)(T_{i})\| \leq 1$.
This means that $\sigma$ is contractive.\end{proof}

\begin{xmps} \label{xmp}
\begin{enumerate}
\item
Let
\[
A := R\{x/r, r_{1}/x\} := \left\{\sum_{-\infty}^\infty a_{n} x^n, a_{n} \in R\ \mathrm{and}
\left\{\begin{array} {l} \|a_{n}\|r^n \to 0\ \mathrm{when}\ n \to \infty \\  \|a_{n}\|r_{1}^n \to 0\ \mathrm{when}\ n \to -\infty \end{array}\right.
\right\}
\]
with $0 < r_{1} \leq r$.
This is the ring of functions on the closed relative annulus and it comes with the usual norm
\[
\left\| \sum_{-\infty}^\infty a_{n} x^n \right\| := \max \{ \|a_{n}\|r^n, n \geq 0 ; \|a_{n}\|r_{1}^n, n \leq 0\}.
\]
Giving an $R$-algebra endomorphism $\sigma$ of $A$ amounts to giving an element $\sigma(x) = \sum_{-\infty}^\infty a_{n} x^n \in A$ with the condition $r_{1} \leq \|\sigma(x)\| \leq r$, which can be rewritten
\[
r_{1} \leq  \max \{ \|a_{n}\|r^n, n \geq 0 ; \|a_{n}\|r_{1}^n, n \leq 0\} \leq r.
\]
\item
As a particular case, we may look for an endomorphism of the form $\sigma(x) = qx+h$ with $q, h \in R$.
Then, the conditions become
\begin{equation} \label{cond1}
\|q\| \leq 1 \quad \mathrm{and} \quad  \|h\| \leq r,
\end{equation}
and
\begin{equation} \label{cond2}
\mathrm{either}\quad \|q\| \geq \frac{r_{1}}r  \quad \mathrm{or} \quad \|h\| \geq r_{1}.
\end{equation}
\item Note that such an endomorphism need not be bijective.
Actually, this happens exactly when $q \in R^\times$ and
\[
\|q\| = \|q^{-1}\| = 1 \quad \mathrm{and} \quad  \|h\| \leq r.
\]
\end{enumerate}
\end{xmps}

\subsection{Twisted modules}

As we did in \cite{LeStumQuiros18b} and \cite{LeStumQuiros18} in the algebraic situation, one can extend many constructions on affinoid $R$-algebras (case $\sigma = \mathrm{Id}_{A}$) to analog constructions on \emph{twisted} affinoid $R$-algebras (with non trivial endomorphism $\sigma$).
We will then use the attribute \emph{twisted} (or the prefix ``$\sigma$-'') and replace it with the word \emph{usual} (or remove the prefix) in the case $\sigma = \mathrm{Id_{A}}$.
Also, we will try to make the notations as light as possible by not mentioning $A$ and $R$ but we will add the subscript ``$A/R$''  when necessary.

As a first example, we may consider \emph{twisted $A$-modules}, which are \emph{finite} $A$-modules $M$ endowed with a $\sigma$-linear endomorphism $\sigma_{M}$.
Note that $\sigma_{M}$ is automatically continuous since it splits as $M \to \sigma^*M \to M$.
We will denote by $\sigma\mathrm{-Mod}(A/R)$ the category of twisted $A$-modules.
In order to better understand these objects, it is convenient to introduce the \emph{twisted polynomial ring} $A[T]_{\sigma}$: this is the non-commutative polynomial ring with the commutation rule $Tz=\sigma(z)T$.
Then, there exists an equivalence of categories ($\sigma_{M}$ acts as multiplication by $T$):
\[
\sigma\mathrm{-Mod}(A/R) \simeq A\mathrm{-finite}\ A[T]_{\sigma}\mathrm{-Mod}.
\]
Moreover, it induces an isomorphism:
\[
\mathrm H^*_{\sigma}(M) := \mathrm H^*[M \overset {1 - \sigma_{M}} \longrightarrow M] \simeq \mathrm{Ext}^*_{A[T]_{\sigma}}(A, M).
\]

\subsection{Twisted derivations}

A \emph{twisted derivation} on $A$ with values in a finite $A$-module $M$ is an $R$-linear map $D : A \to M$ satisfying the \emph{twisted Leibniz rule}:
\[
\forall z_{1}, z_{2} \in A, \quad D(z_{1}z_{2}) = z_{2}D(z_{1}) + \sigma(z_{1})D(z_{2}).
\]

\begin{prop} \label{contprop}
A twisted derivation $D$ on a twisted affinoid $R$-algebra $A$ is automatically continuous.
\end{prop}

\begin{proof}
Let us fix a surjection $\pi : R\{\underline T\} \twoheadrightarrow A$ (we use the standard multiindex notation) and endow $A$ with the quotient norm.
We set $\widetilde D = D \circ \pi$ and $\widetilde \sigma = \sigma \circ \pi$.
If we denote by $D_{0}$ the restriction of $\widetilde D$ to the polynomial ring $R[\underline T]$, then we have $\|D_{0}\| \leq \|\widetilde \sigma\| \max_{i=1}^N \|D_{0}(T_{i})\| < +\infty$.
This is proved by induction on the total degree:
$$
\|D_{0}(\underline T^{\underline n}T_{i})\| = \|\pi(T_{i})D_{0}(\underline T^{\underline n}) + \widetilde \sigma(\underline T^{\underline n})D_{0}(T_{i})\| \leq  \max \{\|D_{0}(\underline T^{\underline n})\|, \|\widetilde \sigma\| \|D_{0}(T_{i})\| \}.
$$
It follows that $D_{0}$ is continuous and extends therefore uniquely to a continuous map $\widehat D_{0} : R\{\underline T\} \to D$.
We only have to show now that $\widetilde D = \widehat D_{0}$, or equivalently, if we set $E := \widetilde D - \widehat D_{0}$, that $E = 0$.
Recall that if $\mathfrak m$ is a maximal ideal of $R\{\underline T\}$ and $\mathfrak m_{0} := \mathfrak m \cap R[\underline T]$, then we have $\mathfrak m = \mathfrak m_{0}R\{\underline T\}$ and for all $n \in \mathbb N$, $R\{\underline T\} = R[\underline T] + \mathfrak m^n$.
Thus, if $F \in R\{\underline T\}$, we can write $F = P + QG$ with $P, Q \in R[\underline T]$ and $Q \in \mathfrak m^n$.
Therefore, since the restriction of $E$ to $R[\underline T]$ is zero, we see that
$$
E(F) = E(P) + \pi(Q)E(G) + \widetilde \sigma(G)E(Q) = \pi(Q)E(G).
$$
Thus, if we consider $M$ as an $R\{\underline T\}$-module via $\pi$, we have $E(F) \in \mathfrak m^nM$.
This being true for all $\mathfrak m$ and all $n$, and $M$ being finite over $R\{\underline T\}$, it implies that $E(F) = 0$.
 \end{proof}

We will denote by $\mathrm{Der}_{\sigma}(A,M)$ the submodule of all twisted derivations on $A$ with values in the finite $A$-module $M$ and   we will also write $\mathrm T_{A,\sigma} := \mathrm{Der}_{\sigma}(A,A)$.

One can build on this notion and define for a finite $A$-module $M$ the notion \emph{twisted derivation} $D$ of $M$: this is an $R$-linear map $D : M \to M$ satisfying the \emph{twisted Leibniz rule}: there exists a twisted derivation\footnote{A twisted derivation $D$ of $A$ is not uniquely determined by the twisted derivation of $M$: this only holds up to the annihilator of $M$.}
(that we also denote by $D$) of $A$ such that
\[
\forall z \in A, \forall s \in M, \quad D(zs) = D(z)s + \sigma(z)D(s).
\]
One can check as in proposition \ref{contprop} that such a twisted derivation is automatically continuous.
We will denote by $\mathrm{Der}_{\sigma}(M)$ the submodule of all twisted derivations of $M$.

We may also consider the ring of \emph{small twisted differential operators} $\overline {\mathrm D}_{ \sigma}$ which is the smallest subring of $\mathrm{End}_{R\mathrm{-cont}}(A) := \mathrm{Hom}_{R\mathrm{-cont}}(A,A)$ that contains both $A$ and $\mathrm{T}_{A,\sigma}$.
As an example, the endomorphism $\sigma$ itself is a small twisted differential operator (because $1 - \sigma$ is a twisted derivation).
If $M$ is a $\overline {\mathrm D}_{ \sigma}$-module, then it has a natural action by twisted derivations, but such an action does \emph{not} lift in general (curvature \emph{and} $p$-curvature phenomena).

\subsection{Linearization}

We will denote by $P := A \widehat \otimes_{R} A$ the completed tensor product of $A$ by itself and by $I$ the ideal of the multiplication map $P \to A$.
We will see $P$ as an $A$-module through the action on the \emph{left} and simply write $z := z \otimes 1$.
By contrast, we will write $\widetilde z = 1 \otimes z$ and call \emph{Taylor map} the morphism
 \[
 \theta : A \to P, \quad z \mapsto \widetilde z
 \]
that defines the action on the right.
Recall that $I$ is generated as an ideal by the image of the \emph{differentiation map}
 \[
 \mathrm d : A \to P, \quad z \mapsto \widetilde z - z
 \]
(which is also continuous).
We will also need below the \emph{comultiplication map}
 \[
 \xymatrix@R0cm{\delta : A \widehat \otimes_{R} A \ar[r] & A \widehat \otimes_{R} A \widehat \otimes_{R} A
 \\ z_{1} \otimes z_{2} \ar[r] &z_{1} \otimes 1 \otimes z_{2}}
 \]
that will be seen as a map $\delta : P \to P \widehat \otimes'_{A} P$ (the $\otimes'$ means that we  use the \emph{right} action on the left factor - we will keep this convention throughout).

If $M, N$ are two finite $A$-modules, then $\mathrm{Hom}_{R\mathrm{-cont}}(M, N)$ is naturally a $P$-module for $z_{1}\widetilde z_{2} \cdot \varphi = z_{1} \circ \varphi \circ z_{2}$ (where $z_{i}$ stands here for multiplication by $z_{i}$).
One can remark that $\mathrm d z \cdot \varphi = [\varphi, z]$ is then the commutator.
Moreover, there exists a \emph{linearization} isomorphism
\[
\mathrm{Hom}_{R\mathrm{-cont}}(M, N) \simeq \mathrm{Hom}_{A\mathrm{-cont}}(P \otimes'_{A} M, N), \quad \varphi \mapsto \widetilde \varphi
\]
with the relations
\[
\varphi(s) = \widetilde \varphi(1 \otimes' s) \quad \mathrm{and} \quad \widetilde \varphi (f \otimes' s) = (f \cdot \varphi)(s).
\]
If $\varphi : M \to N$ and $\psi : L \to M$ are two continuous $R$-linear maps between finite $A$-modules, then we have
\[
\xymatrix{\widetilde {\varphi \circ \psi}:  & P \otimes'_{A} L \ar[r]^-{\delta} & P \widehat \otimes'_{A} P  \otimes'_{A} L \ar[r]^-{\mathrm{Id} \otimes \widetilde \psi} & P \otimes'_{A} M \ar[r]^-{ \widetilde \varphi} & N.
}
\]

\subsection{Twisted principal parts}

We extend $\sigma$ to $P$ by setting $\sigma_{P} := \sigma_{A} \widehat\otimes_{R} \mathrm{Id}_{A}$.
Then, we consider for $n \in \mathbb N$, the \emph{twisted product}
 \[
 I^{(n)} := I\sigma(I)\cdots \sigma^{n-1}(I)
 \]
 (as ideals) and the \emph{ring of twisted principal parts of order $n$}: $P_{(n)} := P/I^{(n+1)}$
 (we will write $(n)_{\sigma}$ in place of $(n)$ if we want to make clear the dependence on $\sigma$ and remove the parenthesis when $\sigma = \mathrm{Id}_{A}$).
We also introduce the \emph{module of twisted differential forms} $\Omega^1_{\sigma} = I/I^{(2)}$ and the \emph{ring of (all) twisted principal parts} $\widehat P_{\sigma} := \varprojlim P_{(n)}$.

\begin{prop} \label{finit}
If $A$ is a twisted affinoid $R$-algebra, then $P_{A/R,(n)_{\sigma}}$ is a finite $A$-module (for the left structure).
If we assume that $\sigma$ is finite, then $P_{A/R,(n)_{\sigma}}$ is also a finite $A$-module for the right structure.
\end{prop}

\begin{proof} 
One first checks that there exists, for each $n \in \mathbb N$, a short exact sequence
\begin{equation} \label{sigmn}
\xymatrix@R0cm{0 \ar[r]& \sigma^n(I) \ar[r]& P \ar[r]& A \ar[r]& 0 \\ && z_{1}\widetilde z_{2} \ar[r] & z_{1} \sigma^n(z_{2})}.
\end{equation}
Since $I^{(n)}$ is finitely generated as a $P$-module, it follows that $I^{(n)}/I^{(n+1)} := I^{(n)}/I^{(n)}\sigma^n(I)$ is finitely generated as an $A$-module.
We may then proceed by induction on $n$.
If we consider now the action of $A$ on the right on $P$, then we need to endow $A$ with its $A$-module structure via $\sigma^n$ in order to keep $A$-linearity in the exact sequence \eqref{sigmn}.
We may then proceed in the same way in order to show that $P_{(n)}$ is also finite for its right structure as long as $\sigma$ itself is finite.
 \end{proof}

Since we need it below, we also want to mention that we always have
\[
\sigma^n(\mathrm dz) \cdot \varphi = [\varphi, z]_{\sigma^n} := \varphi \circ z - \sigma^n(z) \circ \varphi
\]
so that $\|\sigma^n(\mathrm dz) \cdot \varphi\| \leq \| \varphi \|\|z \|$ when $\sigma$ is contractive.

\subsection{Twisted connections}

It follows from proposition \ref{finit} that $\Omega^1_{\sigma}$ is a finite $A$-module.
The (continuous) differentiation map $\mathrm d : A \to P$ takes values inside $I$ and we will still denote by $\mathrm d : A \to \Omega^1_{\sigma}$ the composition with the projection.
This is the universal continuous twisted derivation in the sense that it induces an isomorphism
\[
\mathrm{Hom}_{A}(\Omega^1_{\sigma}, M) \simeq \mathrm{Der}_{\sigma}(A,M)
\]
for any finite $A$-module $M$.

A \emph{twisted connection} on a finite $A$-module $M$ is an $R$-linear map $\nabla : M \to M \otimes_{A} \Omega^1_{\sigma}$ that satisfies the twisted Leibniz rule
\[
\forall z \in A, \forall s \in M, \quad \nabla(zs) = s \otimes d(z) + \sigma(z)\nabla(s).
\]
A twisted connection is automatically continuous because it is induced by a $P_{(1)}$-linear map $\epsilon : P_{(1)} \otimes'_{A} M \to M \otimes_{A} P_{(1)}$ in the sense that $\epsilon(1 \otimes' s) = s \otimes 1 + \nabla(s)$.
Finite $A$-modules endowed with a twisted connection form a category $\nabla_{\sigma}\mathrm{-Mod}(A/R)$ (with \emph{horizontal maps} as morphisms).
Any twisted connection on an $A$-module $M$ gives rise to a linear action by twisted derivations.
Moreover, if $\Omega^1_{\sigma}$ is a flat $A$-module, then we obtain an equivalence  that preserves cohomology.
This is the case in practice, and we will use the vocabulary of $\nabla_{\sigma}$-modules, which is more pleasant.

\subsection{Twisted differential operators}

Let $M, N$ be two finite $A$-modules.
A \emph{twisted differential operator (of infinite radius) of degree at most $n$} is an $R$-linear map $\varphi : M \to N$ whose $A$-linearization $\widetilde \varphi : P \otimes'_{A} M\to N$  factors through $P_{(n)} \otimes'_{A} M$.
Since $P_{(n)}$ is a finite $A$-module, a twisted differential operator is automatically continuous.
We will then denote the induced map by $\widetilde \varphi_{n} : P_{(n)} \otimes'_{A} M \to N$.
We obtain for each $n \in \mathbb N$, a finite $A$-module
\[
\mathrm{Diff}^{(\infty)}_{\sigma,n}(M, N) \simeq \mathrm{Hom}_{A}(P_{(n)} \otimes'_{A} M, N)
\]
and we will write
 \[
 \mathrm{Diff}^{(\infty)}_{\sigma}(M, N) = \cup_{n \in \mathbb N} \mathrm{Diff}^{(\infty)}_{\sigma,n}(M, N).
 \]
Note that, by definition, the action of $P$ on $\mathrm{Hom}_{R\mathrm{-cont}}(M, N)$ induces an action of $\widehat P_{\sigma}$ on $\mathrm{Diff}^{(\infty)}_{\sigma}(M, N)$.
Actually, a continuous $R$-linear map $\varphi$ falls into $\mathrm{Diff}^{(\infty)}_{\sigma,n}(M, N)$ if and only if $\sigma^n(I) \cdot \varphi \subset \mathrm{Diff}^{(\infty)}_{\sigma,n-1}(M, N)$.
Twisted differential operators are stable under composition and this turns $\mathrm D_{\sigma}^{(\infty)} := \mathrm{Diff}^{(\infty)}_{\sigma}(A, A)$ into a subring of $\mathrm{End}_{R\mathrm{-cont}}(A)$ containing $\overline {\mathrm D}_{\sigma}$.
In particular, any $\mathrm D_{\sigma}^{(\infty)}$-module gives rise to a $\overline {\mathrm D}_{\sigma}$-module.
But again, this is not an equivalence in general even for finite $A$-modules ($p$-curvature problem).
Note also that the action of $\mathrm D_{\sigma}^{(\infty)}$ on an $A$-finite $\mathrm D_{\sigma}^{(\infty)}$-module $M$ induces a continuous $\widehat P_{\sigma}$-linear ring homomorphism 
 \[
 \mathrm D_{\sigma}^{(\infty)} \to \mathrm{Diff}^{(\infty)}_{\sigma}(M,M).
 \]
 
\subsection{Twisted Taylor structures}

One can show that the comultiplication map $\delta : P \to P \widehat \otimes'_{A} P$ induces for all $m,n \in \mathbb N$ a comultiplication map
\[
\delta_{(m),(n)} : P_{(m+n)} \to P_{(m)} \otimes'_{A} P_{(n)}.
\] 
A \emph{twisted Taylor structure} on a finite $A$-module $M$ is a family of compatible \emph{twisted Taylor maps} $\theta_{n} : M \to M \otimes_{A} P_{(n)}$.
More precisely, we require for $\theta_{n}$ to be $A$-linear when the target is endowed with the $A$-module structure coming from the \emph{right} structure of $P_{(n)}$, and compatibility means that
\[
\forall m, n \in \mathbb N, \quad (\theta_{n} \otimes \mathrm{Id}_{P_{(m)}}) \circ \theta_{m} = (\mathrm{Id}_{M} \otimes \delta_{(n),(m)}) \circ \theta_{m+n}.
\]
We shall then write
\[
\widehat \theta := \varprojlim \theta_{n} : M \mapsto M \otimes_{A} \widehat P_{\sigma}
\]
and call $\widehat \theta(s)$ the \emph{twisted Taylor series} of $s \in M$.
A twisted Taylor structure on $M$ induces a $D^{(\infty)}_{\sigma}$-module structure via
\[
\varphi s = \sum_{k=0}^{n} \widetilde \varphi(f_{k}) s_{k} \quad \mathrm{if} \quad \theta_{n}(s) = \sum_{k=0}^{n} s_{k} \otimes f_{k},
\]
when the twisted differential operator $\varphi$ has order $n$.
Moreover, we obtain an equivalence of categories when all $P_{(n)}$'s are flat $A$-modules (see section 5 of \cite{LeStumQuiros18}).

\subsection{Base change}

We will need below the following result:

\begin{prop} \label{extn}
If $A \to B$ is any morphism of twisted affinoid $R$-algebras, then there exists for each $n \in \mathbb N$, a natural morphism
\[
B \otimes_{A} P_{A/R,(n)_{\sigma}} \to P_{B/R,(n)_{\sigma}}.
\]
When $B = A\{\frac {f_{1}}{f_{0}}, \ldots, \frac {f_{r}}{f_{0}}\}$, this is an isomorphism.
\end{prop}

\begin{proof}
The existence of the morphism is obtained by functoriality and the second assertion may be seen as a consequence of the geometric nature of the construction.
More precisely, one may consider the rigid analytic variety $X$ associated to $A$ as well as the endomorphism $s$ of $X$ corresponding to $\sigma$.
If $\mathcal I_{X}$ denotes the ideal of the diagonal embedding $\iota_{X} : X \hookrightarrow X \times_{R} X$, one sets $\mathcal I_{X}^{(n)} := \mathcal I_{X}s^{*}(\mathcal I_{X})\cdots s^{n-1*}(\mathcal I_{X})$ and $\mathcal P_{X,(n)} = \iota_{X}^{-1}(\mathcal O_{X \times_{R} X}/\mathcal I_{X}^{(n+1)})$.
Recall that $A\{\frac {f_{1}}{f_{0}}, \ldots, \frac {f_{r}}{f_{0}}\}$ is the affinoid algebra associated to the open subset defined by $|f_{1}(x)|, \ldots, |f_{r}(x)| \leq |f_{0}(x)|$.
If we let $j : U \hookrightarrow X$ denote the inclusion map, we will have $j^{-1}\mathcal P_{X,(n)} = \mathcal P_{U,(n)}$ and taking global sections provides the expected isomorphism.
Details are left to the reader.
\end{proof}

In the situation of proposition \ref{extn}, there exists a natural morphism
\[
B \otimes_{A} \Omega^1_{A/R,\sigma} \to \Omega^1_{B/R,\sigma}
\]
which is an isomorphism when $B = A\{\frac {f_{1}}{f_{0}}, \ldots, \frac {f_{1}}{f_{0}}\}$.
In this last case, we also have an isomorphism
\[
B \otimes_{A} D^{(\infty)}_{A/R,\sigma} \simeq D^{(\infty)}_{B/R,\sigma}.
\]

\subsection{Twisted coordinate}

We call $x \in A$ a \emph{twisted coordinate} if all the maps
\begin{equation} \label{basdf}
\xymatrix@R0cm{A[\xi]_{\leq n} \ar[r] & P_{(n)} \\ \xi \ar@{|->}[r] & \overline {\widetilde x - x}}
\end{equation}
(where the left hand side denotes the set of polynomials of degree at most $n$) are bijective.
We call $x$ a \emph{quantum coordinate} if, moreover, there exists $q,h \in R$ such that $\sigma(x) = qx + h$.
We will say \emph{$q$-coordinate} when we want to make explicit the value of $q$.

\begin{xmps}
\begin{enumerate}
\item
If $A$ is formally smooth over $R$ and $\sigma = \mathrm{Id}_{A}$, then a quantum coordinate is the same thing as an \'etale coordinate.
\item
It follows from proposition \ref{extn} that if $A$ is the ring of functions on an affinoid domain $X$ of the $R$-line, $x$ is a (usual) coordinate on the $R$-line, and $\sigma(x)$ is a convergent power series on $X$, then $x$ is a twisted coordinate on $A$ (but we still have to check that $X$ is stable under $\sigma$).
In the particular case $\sigma(x) = qx + h$, this is a quantum coordinate.
\item
This applies to a \emph{standard affinoid domain} as in \cite{Pulita08}: a union of subsets of the form
\begin{equation} \label{pulaf}
X := \mathbb D^+(c,r) \setminus \left( \cup_{i=1}^h \mathbb D^-(c_{i}, r_{i}) \right).
\end{equation}
In other words, $A$ will be a \emph{standard affinoid algebra}: a product of rings of the form
\[
R\left\{\frac {x-c}r, \frac {r_{1}}{x-c_{1}}, \ldots, \frac {r_{h}}{x-c_{h}}\right\}.
\]
When $R = K$ is an algebraically closed complete ultrametric field, any affinoid domain of the line is standard.
\item The simplest non trivial case is an annulus centered at the origin
\[
X := \mathbb D^+(0,r) \setminus \mathbb D^-(0, r_{1})
\]
and we fall back onto our above example $A := R\{x/r, r_{1}/x\}$.
\end{enumerate}
\end{xmps}

Alternatively, one can show that $x$ is a twisted coordinate on $A$ if and only if there exists canonical isomorphisms of $A$-algebras
\begin{equation} \label{dfiso}
A[\xi]/ \xi^{(n)} \simeq P_{(n)} \quad \mathrm{with} \quad \xi^{(n)} := \prod_{i=0}^{n-1} (\xi + x - \sigma^n(x)).
\end{equation}
When $x$ is a twisted coordinate on $A$, then $\Omega^1_{\sigma}$ is free of dimension one generated by the image of $\xi$ and there exists therefore a unique twisted derivation $\partial_{A,\sigma}$ on $A$ such that $\partial_{A,\sigma}(x) = 1$.
In particular, an action by twisted derivations on a finite $A$-module $M$ is simply given by an $R$-linear endomorphism $\partial_{M,\sigma}$ satisfying
\[
\forall z \in A, \forall s \in M, \quad \partial_{M,\sigma}(zs) = \partial_{A,\sigma}(z)s + \sigma(z)\partial_{M,\sigma}(s).
\]

\subsection{Twisted Weyl algebra}

When $x$ is a twisted coordinate, the \emph{Ore extension} $\mathrm D_{\sigma}$ of $A$ by $\sigma$ and $\partial_{A,\sigma}$ is called the \emph{twisted Weyl algebra} of $A$:
this is the free $A$-module (of infinite rank) with formal basis $\partial_{\sigma}^{k}$ and commutation rule
\[
\partial_{\sigma} \circ z = \partial_{A,\sigma}(z) + \sigma(z)\partial_{\sigma}.
\]
Note that it depends on $x$ because $\partial_{A,\sigma}$ does.
One easily sees that here exists an equivalence of categories
\[
\nabla_{\sigma}\mathrm{-Mod}(A/R) \simeq A\mathrm{-finite}\ \mathrm D_{A/R,\sigma}\mathrm{-Mod}.
\]
Moreover, it induces an isomorphism:
\[
\mathrm H^*_{\partial_{\sigma}}(M) := \mathrm H^*[M \overset {\partial_{\sigma}} \longrightarrow M] \simeq \mathrm{Ext}^*_{D_{\sigma}}(A, M).
\]

On the other hand, there exists a canonical map
\begin{equation} \label{strong}
A[T]_{\sigma} \to \mathrm D_{\sigma}, \quad T \mapsto 1 - (x - \sigma(x))\partial_{\sigma}.
\end{equation}
We call $x$ \emph{strong} when $x - \sigma(x) \in A^\times$ (note that, this cannot happen when $\sigma = \mathrm{Id}_{A}$).
If this is the case, then the morphism \eqref{strong} is an isomorphism.
It follows that, in general, there exists a functor
\[
\nabla_{\sigma}\mathrm{-Mod}(A/R) \to \sigma\mathrm{-Mod}(A/R)
\]
ant that, when $x$ is strong, this is an equivalence providing an isomorphism
\[
H^*_{\sigma}(M) = H^*_{\partial_{\sigma}}(M).
\]
To summarize, when $\sigma \neq \mathrm{Id}_{A}$, there is essentially no difference between an $A[T]_{\sigma}$-module, a twisted module, a module with an action by twisted derivations, a module with a twisted connection or a $D_{\sigma}$-module (and these identifications are compatible with cohomology).
We will see later that, with some extra conditions, we can add $D^{(\infty)}_{\sigma}$-modules, and even what we will call $D^{(\eta)}_{\sigma}$-modules, to the list.

\subsection{Standard twisted differential operators}

When $x$ is a twisted coordinate on $A$, we have
\begin{equation} \label{duald}
\widehat P_{\sigma} \simeq A[[\xi]]_{\sigma} := \varprojlim A[\xi]/\xi^{(n)} \quad \mathrm{and} \quad \mathrm D_{\sigma}^{(\infty)} \simeq \varinjlim_{n} \mathrm{Hom}_{A}(A[\xi]/\xi^{(n)}, A).
\end{equation}
We will denote by $\{\partial_{\sigma}^{[k]}\}_{k \in \mathbb N}$ the dual basis to $\{\xi^{(k)}\}_{k \in \mathbb N}$ and call $\partial_{\sigma}^{[k]}$ the \emph{standard twisted differential operator} of order $k$.
Note that the action of $\widehat P_{\sigma}$ on $\mathrm D_{\sigma}^{(\infty)}$ composed with evaluation at $1$ gives back the paring of $A$-modules
\[
\xymatrix@R0cm{\widehat P_{\sigma} \times  \mathrm D_{\sigma}^{(\infty)} \ar[r] & \mathrm D^{(\infty)}_{\sigma} \ar[r] &  A
\\ (\xi^{(n)},\partial_{\sigma}^{[k]}) \ar@{|->}[rr] && 1 \ \mathrm{if} \ n = k \ \mathrm{and} \ 0 \ \mathrm{otherwise}.
}
\]
One can also prove explicit identities such as
\begin{equation} \label{dkx}
\forall k \in \mathbb N \setminus\{0\}, \quad \partial_{\sigma}^{[k]} \circ x = \sigma(x) \partial_{\sigma}^{[k]} + \partial_{\sigma}^{[k-1]},
\end{equation}
from which we can derive
\begin{equation} \label{actdk}
\forall k \in \mathbb N \setminus\{0\}, \quad \xi \cdot \partial_{\sigma}^{[k]} = \partial_{\sigma}^{[k-1]} -(x -\sigma(x)) \partial_{\sigma}^{[k]}.
\end{equation}

Also, if $M$ is an $A$-finite $\mathrm D_{\sigma}^{(\infty)}$-module and $s \in M$, then its twisted Taylor series will be given by
\[
\widehat \theta(s) = \sum_{k=0}^{\infty} \partial^{[k]}_{\sigma}(s) \otimes \xi^{(k)} \in M \otimes_{A} \widehat P_{\sigma},
\]
providing an a posteriori justification for the terminology.

 When $x$ is actually a  $q$-coordinate, there exists an epi-mono factorization
\[
\mathrm D_{\sigma} \twoheadrightarrow \overline {\mathrm D}_{\sigma} \hookrightarrow \mathrm D^{(\infty)}_{\sigma}
\]
sending $\partial_{\sigma}^{k}$ to $(k)_{q}!\partial_{\sigma}^{[k]}$.
When all positive $q$-integers are invertible (for example if $q \in K$ is not a root of unity), then we get equalities $D_{\sigma} = \overline D_{\sigma} = D_{\sigma}^{(\infty)}$.
In particular, in this situation, we obtain a sequence of equivalences
\[
\nabla_{\sigma}\mathrm{-Mod}(A/R) \simeq A\mathrm{-finite}\ \mathrm D_{A/R,\sigma}\mathrm{-Mod} \simeq A\mathrm{-finite}\ \mathrm D^{(\infty)}_{A/R,\sigma}\mathrm{-Mod}
\]
inducing isomorphisms
\[
H^*_{\partial_{\sigma}}(M) = \mathrm{Ext}_{\mathrm D_{\sigma}}^*(A, M) = \mathrm{Ext}_{\mathrm D_{\sigma}^{(\infty)}}^*(A, M).
\]
When $q$ \emph{is} a root of unity, this is not true anymore.

\subsection{Formal density}

When $x$ is a twisted coordinate, we denote by $K^{[k]}$ the submodule (freely) generated by all $\partial^{[i]}_{\sigma}$ for $i \geq k$, and we consider the completion
\[
\widehat {\mathrm D}_{\sigma}^{(\infty)} =: \varprojlim \mathrm D_{\sigma}^{(\infty)}/K^{[k]}
\]
which is \emph{not} a ring in general.
Note however that formula \eqref{actdk} turns $\widehat {\mathrm D}_{\sigma}^{(\infty)}$ into an $A[\xi]$-module.
The \emph{formal density lemma} tells us that the last isomorphism of \eqref{duald} extends to
\begin{equation} \label{denslem}
\xymatrix@R0cm{\widehat {\mathrm D}_{\sigma}^{(\infty)} \ar[r]^-{\simeq} & \mathrm{Hom}_{A}(A[\xi], A)}
\end{equation}
where the right hand side does \emph{not} depend on $\sigma$.
From this result, one can derive the \emph{formal deformation map} and the \emph{formal confluence theorems} as in \cite{LeStumQuiros18} but this is not what we are interested in here.

Recall that we have
\[
\mathrm D_{\sigma}^{(\infty)} := \left\{\sum_{0 \leq k << \infty} z_{k} \partial_{\sigma}^{[k]}, \quad z_{k} \in A \right\} \subset \widehat {\mathrm D}_{\sigma}^{(\infty)} := \left\{\sum_{k=0}^\infty z_{k} \partial_{\sigma}^{[k]}, \quad z_{k} \in A \right\}.
\]
Our goal is to use the topology of $A$ in order to describe some $A$-modules that lie in between, that have a ring structure (such as $\mathrm D_{\sigma}^{(\infty)}$) and, at the same time, are essentially independent of $\sigma$ (such as $\widehat {\mathrm D}_{\sigma}^{(\infty)}$).

\subsection{Quantum analogs}
 
When $x$ is a $q$-coordinate, we have
$$
\forall k,l \in \mathbb N, \quad \partial_{\sigma}^{[k]} \circ \partial_{\sigma}^{[l]} = {k+l \choose l}_{q} \partial_{\sigma}^{[k+l]}.
$$
More generally, $q$-analogs will appear systematically in formulas.

\begin{xmp}
When $\sigma(x) = qx$, we have 
\begin{equation} \label{dkx}
\partial^{[k]}_{\sigma}(x^n) = \left\{ \begin{array}{l} {n \choose k}_{q} x^{n-k} \ \mathrm{for}\ k \leq n \\ 0\ \mathrm{otherwise}. \end{array}\right.
\end{equation}
\end{xmp}

\begin{prop} \label{qest}
If $q \in R$, we have
\[
\left\|{n \choose k}_{q} \right\|\leq \max\left\{1, \|q\|^{k(n-1)}\right\}.
\]
\end{prop}

The case $k = 1$ reads $\|(n)_{q}\| \leq \max\left\{1, \|q\|^{n-1}\right\}$.

\begin{proof}
By induction, we will have for $n, k \geq 1$,
\begin{eqnarray*}
\left\|{n \choose k}_{q}\right\| &\leq& \max\left\{\left\|{n-1 \choose k-1}_{q}\right\|, \|q\|^k \left\|{n-1 \choose k}_{q}\right\| \right\}
\\
&\leq& \max\left\{1, \|q\|^{(k-1)(n-2)}, \|q\|^k, \|q\|^{k+k(n-2)} \right\}
\\
&=& \max\left\{1, \|q\|^{k(n-1)}\right\}. \qedhere
\end{eqnarray*}
\end{proof}

\section{Twisted principal parts of finite radius}

We let $(A, \sigma)$ be a twisted affinoid $R$-algebra with twisted coordinate $x$ and fixed contractive norm.

\begin{dfn} \label{xrad}
The \emph{$x$-radius} of $\sigma$ is
\[
\rho_{x}(\sigma) := \|x -\sigma(x)\|.
\]
\end{dfn}

We will usually drop the subscript $x$ and might even just denote by $\rho$ the $x$-radius of $\sigma$ on $A$.
We will be mostly interested in conditions of the type $\eta \geq \rho$ in which $\eta$ is some non negative real number.

\begin{lem} \label{xrad}
For all $n \in \mathbb N$, we have $\rho(\sigma^{n}) \leq \rho(\sigma)$.
\end{lem}

\begin{proof}
Since $\sigma$ is a contraction, we will have by induction on $n$,
\begin{eqnarray*}
\|x -\sigma^{n+1}(x)\| &\leq& \max \{\|x -\sigma^{n}(x)\|, \|\sigma^{n}(x -\sigma(x))\| \}
\\
&\leq& \max \{\rho(\sigma^n), \rho(\sigma) \}
\\
&=& \rho(\sigma). \qedhere
\end{eqnarray*}\end{proof}

\begin{xmps}
\begin{enumerate}
\item We have $\rho(\mathrm{Id}_{A}) = 0$ and it follows that the condition $\eta \geq \rho$ is always satisfied in this case.
\item When $A := R\{x/r, r_{1}/x\}$ and $\sigma(x) = qx+h$ (so that conditions \eqref{cond1} and \eqref{cond2} in the second example of \ref{xmp} are satisfied), we have
\[
x - \sigma(x) = (1-q)x - h
\]
and it follows that
\[
\rho(\sigma) =\max\{\|1-q\|r, \| h \| \}.
\]
Thus, we see that $\eta \geq \rho$ if and only if
\begin{equation} \label{cond3}
\|1-q\| \leq \frac \eta r \quad \mathrm{and} \quad \|h\| \leq \eta.
\end{equation}
Note that the first (resp. second) condition will always be satisfied in the difference (resp. $q$-difference) equation case.
\end{enumerate}
\end{xmps}

We consider now the affinoid algebra
\[
A\{\xi/\eta\} := \left\{\sum_{n=0}^\infty z_{n} \xi^n, z_{n} \in A\ \mathrm{and}\ \|z_{n}\|\eta^n \to 0\ \mathrm{when}\ n \to \infty \right\}
\]
of functions that converge on the relative closed disk of radius $\eta$.
This is a Banach algebra for the sup norm
\[
\left\|\sum_{n=0}^\infty z_{n} \xi^n \right\|_{\eta} = \max \|z_{n}\|\eta^n.
\]

\begin{rmk}
If $\eta \geq \rho$, then the $\sigma$-linear morphism of rings
\[
\xymatrix@R0cm{ \xi  \ar@{|->}[r] & \xi + \sigma(x) -x}
\]
turns $A\{\xi/\eta\}$ into a twisted affinoid $R$-algebra because
\[
\|\xi + \sigma(x) -x\|_{\eta} = \max\{\eta, \rho\} = \eta.
\]
We will still denote this map by $\sigma$ so that $\sigma(\xi) =\xi + \sigma(x) -x$.
\end{rmk}

\begin{lem} \label{aut}
If $\eta \geq \rho$, then the $A$-linear map
\[
\xymatrix@R0cm{ \xi^n  \ar@{|->}[r] & \xi^{(n)}} := \xi\sigma(\xi)\ldots \sigma^{n-1}(\xi)
\]
defines an isometric automorphism of the $A$-module $A\{\xi/\eta\}$.
\end{lem}

\begin{proof}
Since we assume that $\eta \geq \rho$, it follows from lemma \ref{xrad} that for all $i \in \mathbb N$, we have $\|x -\sigma^i(x)\| \leq \eta$.
Thus we see that
\begin{equation} \label{shaub}
\xi^{(n)} := \prod_0^{n-1} (\xi + (x - \sigma^i(x))) = \xi^n + f_{n}
\end{equation}
where $f_{n} \in A[\xi]_{<n}$ and $\|f_{n}\|_{\eta} \leq \eta^n$.
Hence, the unique $A$-linear endomorphism of $A[\xi]$ that send $\xi^n$ to $\xi^{(n)}$ is an isometry that preserves the degree.
It extends therefore uniquely to an isometry from $A\{\xi/\eta\}$ onto itself.
\end{proof}

\begin{prop} \label{schau}
When $\eta \geq \rho$, $\{\xi^{(n)}\}_{n \in \mathbb N}$ is an \emph{orthogonal Schauder basis} for the $A$-module $A\{\xi/\eta\}$: any $f \in A\{\xi/\eta\}$ can be uniquely written as a convergent sum
\[
f = \sum_{n=0}^\infty z_{n} \xi^{(n)}
\]
with $z_{n} \in A$ and $\|f\|_{\eta} = \sup\{ \|z_{n}\|\eta^n \} $.
\end{prop}

\begin{proof} Follows from lemma \ref{aut} and the fact that $\{\xi^{n}\}_{n \in \mathbb N}$ is an orthogonal Schauder basis for $A\{\xi/\eta\}$. \end{proof}

There exists an analytic version of the isomorphism \eqref{dfiso}:

\begin{prop} \label{comps}
Assume that $\eta \geq \rho$.
Then, for all $n \in \mathbb N$, there exists an isomorphism of $A$-algebras
\[
\xymatrix@R0cm{  A\{\xi/\eta\}/\xi^{(n)} \ar[r]^-\simeq & P_{(n)}
\\ \xi \ar@{|->}[r] & \overline{\widetilde x - x}.}
\]
\end{prop}

\begin{proof}
Since we know that \eqref{dfiso} holds, we only have to prove that the map
\[
\xymatrix@R0cm{A[\xi]_{\leq n} \ar[r]^-\simeq & A[\xi]/\xi^{(n)} \ar[r] & A\{\xi/\eta\}/\xi^{(n)}
}
\]
is bijective.
But this follows from proposition \ref{schau} (or lemma \ref{aut} if you want).\end{proof}

\begin{cor}
Assume that $\eta \geq \rho$.
Then, if $M$ is any finite $A$-module, there exists a canonical embedding
\[
M \otimes_A A\{\xi/\eta\} \hookrightarrow M \otimes_{A} \widehat P_{\sigma}.
\]
\end{cor}

\begin{proof}
Using proposition \ref{comps}, one easily sees that there exists such a map.
Injectivity means that if $\sum s_{i} \otimes f_{i} \equiv 0 \mod \xi^{(n+1)}$ for all $n \in \mathbb N$, then $\sum s_{i} \otimes f_{i} =0$ in $M \otimes_A A\{\xi/\eta\}$.
But this follows from the fact that the $\xi^{(n)}$ form an orthogonal Schauder basis (and the properties of completed direct sums and tensor products).
\end{proof}

\begin{dfn}
If $\eta \geq \rho$, then the image of $A\{\xi/\eta\}$ in $\widehat P_{\sigma}$ is called the ring of \emph{twisted principal parts of radius (at least) $\eta$}.
\end{dfn}

In practice we will identify $A\{\xi/\eta\}$ with its image and call this ring itself the ring of twisted principal parts of radius $\eta$, exactly as we identify $A[[\xi]]_{\sigma}$ with $\widehat P_{\sigma}$.

\section{Radius of convergence}

We still let $(A, \sigma)$ be a twisted affinoid $R$-algebra with fixed contractive norm and $x$ a be twisted coordinate on $A$.
We denote by $\rho$ the $x$-radius of $\sigma$ on $A$ and by $\partial^{[k]}_{\sigma}$ the standard twisted differential operator of order $k$ associated to $x$.

\begin{dfn}
Let $M$ be an $A$-finite $\mathrm D_{\sigma}^{(\infty)}$-module and $\eta \in \mathbb R$.
\begin{enumerate}
\item Let $s \in M$.
Then,
\begin{enumerate}
\item 
$s$ is \emph{$\eta$-convergent} if
\[
\|\partial^{[k]}_{\sigma}(s)\|\eta^k \to 0\ \mathrm{when}\ k \to +\infty,
\]
\item
the \emph{radius of convergence of $s$} is
\[
\mathrm{Rad}(s) = \sup\{\eta, s\ \mathrm{is}\ \eta\mathrm{-convergent}\}.
\]
\item
$s$ is \emph{$\eta\dagger$-convergent} if $\mathrm{Rad}(s) \geq \eta$.
\end{enumerate}
\item
\begin{enumerate}
\item
$M$ is \emph{$\eta$-convergent} if all $s \in M$ are $\eta$-convergent,
\item
the \emph{radius of convergence of $M$} is
\[
\mathrm{Rad}(M) = \inf_{s \in M} \mathrm{Rad}(s).
\]
\item
$M$ is \emph{$\eta\dagger$-convergent} if all $s \in M$ are $\eta\dagger$-convergent,
\end{enumerate}
\end{enumerate}
\end{dfn}

\begin{rmks}
\begin{enumerate}
\item
Alternatively, $M$ is $\eta$-convergent if and only if
\[
\forall s \in M, \quad \|\partial^{[k]}_{\sigma}(s)\|\eta^k \to 0\ \mathrm{when}\ k \to +\infty,
\]
we have
\[
\mathrm{Rad}(M) = \sup\{\eta, M\ \mathrm{is}\ \eta\mathrm{-convergent}\},
\]
and $M$ is \emph{$\eta\dagger$-convergent} if and only if $\mathrm{Rad}(M) \geq \eta$.
\item
We have explicit formulas
\[
\mathrm{Rad}(s) = \underline \lim_{k} \|\partial_{\sigma}^{[k]}(s)\|^{-\frac 1k} \quad \mathrm{and} \quad \mathrm{Rad}(M) = \inf_{s \in M} \underline\lim_{k} \|\partial_{\sigma}^{[k]}(s)\|^{-\frac 1k}.
\]
\item
If $s$ (resp. $M$) is $\eta$-convergent then $s$ (resp. $M$) is $\eta\dagger$-convergent.
There exists a partial converse: $s$ (resp. $M$) is $\eta\dagger$-convergent if and only if $s$ (resp. $M$) is $\eta'$-convergent whenever and $\eta' < \eta$.
\item When $\sigma = \mathrm{Id}_{A}$, we recover standard notions (see proposition 4.4.11 of \cite{LeStum07} for example).
More precisely, if $\mathcal X$ is a smooth affine formal $\mathcal V$-scheme with an \'etale coordinate $x$, $\mathcal M$ is a coherent $\mathcal O_{\mathcal X_{K}}$-module and $M := \Gamma(\mathcal X, \mathcal M)$, then a connection on $\mathcal M$ is convergent (in the sense of rigid cohomology) if and only if it is $1\dagger$-convergent in our sense.
\item This notion of radius of convergence should not be confused with the notion of \emph{radius of convergence at a generic point} which is more subtle.
For example, if $A = R\{x, r_{1}/x\}$ with $r_{1} < 1$ and $\sigma = \mathrm{Id}_{A}$, we have $\mathrm{Rad}(A) = r_{1} \neq 1$ (see example below).
But the trivial connection on $A$ must clearly be ``convergent''.
This phenomenon appears because, even if $A$ is smooth, it has \emph{bad} reduction.
\end{enumerate}
\end{rmks}

\begin{xmps}
\begin{enumerate}
\item Assume that $A = R\{x\}$ and $\sigma(x) = qx$ with $\| q\| \leq 1$.
Then, it follows from proposition \ref{qest} that we always have $\|{n \choose k}_{q}\| \leq 1$.
Thus, thanks to formula \eqref{dkx}, we see that
\[
\forall k \leq n \in \mathbb N, \quad \|\partial^{[k]}_{\sigma}(x^n)\| = \left\|{n \choose k}_{q} x^{n-k}\right\| \leq 1
\]
(and $\partial^{[k]}_{\sigma}(x^n) = 0$ for $k > n$).
It follows that if $z = \sum_{n=0}^{\infty} a_{n}x^n \in R\{x\}$, then we have
\[
\|\partial^{[k]}_{\sigma}(z)\| \leq \max \|a_{n} \partial^{[k]}_{\sigma}(x^n)\| \leq \max_{n \geq k} \|a_{n}\|\ \to 0.
\]
An we see that $A$ is $\eta$-convergent for any $\eta \leq 1$.
\item We consider now the situation
\[
R = K, A = K\{x/r, r_{1}/x\} \quad \mathrm{and} \quad \sigma(x) = qx+h
\]
with $0 < r_{1} \leq r$.
We first notice that formula \eqref{dkx} above implies that for all $n \in \mathbb Z$ and $k > 0$, we have
\begin{equation} \label{indf}
\partial^{[k]}_{\sigma}(x^{n+1}) = \sigma^k(x) \partial^{[k]}_{\sigma}(x^n) + \partial^{[k-1]}_{\sigma}(x^n).
\end{equation}
Since we must have $|\sigma^k(x)| \leq r$, we obtain by induction that, for $n \geq 0$, we have $|\partial^{[k]}_{\sigma}(x^n)| \leq r^{n-k}$.
Using equality \eqref{indf} again,  we can also prove that for $n \leq 0$, we have $|\partial^{[k]}_{\sigma}(x^n)| \leq r_{1}^{n-k}$.
More precisely, since we always have $|\sigma^k(x)| \geq r_{1}$, then by induction on $k-n$, we obtain
\[
|\partial^{[k]}_{\sigma}(x^n)| \leq \frac 1{|\sigma^k(x)|}r_{1}^{n-k+1} \leq r_{1}^{n-k}.
\]
It follows that, if $z = \sum_{-\infty}^\infty a_{n}x^n \in A$, we have
\[
|\partial^{[k]}_{\sigma}(z)| \leq \max_{n} \left |a_{n} \partial^{[k]}_{\sigma}(x^n) \right | \leq \frac 1{r_{1}^k}\max_{n}\{|a_{n}|r^n, |a_{n}|r_{1}^n\} = \frac 1{r_{1}^k}|z|.
\]
Thus we see that $A$ is $\eta$-convergent as long as $\eta < r_{1}$ (one can show that we do have $\mathrm{Rad}(A) = r_{1}$).
For further use, note that $A$ is $\eta$-convergent for some $\eta \geq \rho$ if and only if
\[
|h| \leq \eta < r_{1} \quad \mathrm{and} \quad |1 - q| \leq \frac \eta r.
\]
When this is the case, then $\sigma$ is always an (isometric) automorphism.
\end{enumerate}
\end{xmps}

\begin{lem}
Let $M$ be an $A$-finite $\mathrm D_{\sigma}^{(\infty)}$-module and $\eta \geq \rho$.
Then, $s \in M$ is $\eta$-convergent if and only if its twisted Taylor series
\[
\widehat \theta(s) := \sum_{k=0}^\infty \partial^{[k]}_{\sigma}(s) \otimes \xi^{(k)}
\]
falls inside $M \otimes_A A\{\xi/\eta\} \subset M \otimes \widehat P_{\sigma}$.
\end{lem}

\begin{proof}
Immediate from the definitions since the $\xi^{(k)}$'s form an orthogonal Schauder basis of $A\{\xi/\eta\}$ by proposition \ref{schau}.\end{proof}

\begin{prop} \label{cricv}
Let $M$ be an $A$-finite $\mathrm D_{\sigma}^{(\infty)}$-module and $\eta \geq \rho$.
Then, $M$ is $\eta$-convergent if and only if the twisted Taylor map factors as
\[
\xymatrix{M \ar[rr]^-{\widehat \theta} \ar[rrd]_{\theta_{\eta}} && M \otimes_{A} \widehat P_{\sigma} \\ && M \otimes_A A\{\xi/\eta\}. \ar@{^{(}->}[u]& \Box
}
\]
\end{prop}

\begin{dfn}
This morphism $\theta_{\eta}$ is called the \emph{twisted Taylor map of radius $\eta$} of $M$.
\end{dfn}

\begin{rmk}
If $M$ is an $A$-finite $\mathrm D_{\sigma}^{(\infty)}$-module with Taylor map $\theta$, one may set
\[
\|\theta\|_{\eta} := \sup_{k \in \mathbb N, s \in M \setminus \{0\}} \left\{\frac {\|\partial^{[k]}_{\sigma}(s)\|\eta^k}{\|s\|}\right\} \in [1,+\infty].
\]
Then, we see that $M$ is $\eta$-convergent if and only if $\|\theta\|_{\eta} < \infty$, and in that case $\|\theta_{\eta}\| = \|\theta\|_{\eta}$.
In particular, the twisted Taylor map of radius $\eta$ is automatically continuous.
\end{rmk}

\begin{prop} \label{ext}
Being $\eta$-convergent for an $A$-finite $\mathrm D_{\sigma}^{(\infty)}$-module is stable under subobject, quotient and extension.
\end{prop}

\begin{proof}
Since both $A\{\xi/\eta\}$ and $\widehat P_{\sigma}$ are flat $A$-modules, if we are given an exact sequence
\[
0 \to M' \to M \to M'' \to 0
\]
of $A$-finite $\mathrm D_{\sigma}^{(\infty)}$-modules, we will have a commutative diagram with exact lines
\[
\xymatrix{ 0 \ar[r] & M' \ar[d]^-{\widehat \theta'} \ar[r] & M \ar[r] \ar[d]^-{\widehat \theta}  & M'' \ar[r] \ar[d]^-{\widehat  \theta''}  & 0
\\ 0 \ar[r] & M' \otimes \widehat P_{\sigma} \ar[r] & M \otimes \widehat P_{\sigma}\ar[r] & M'' \otimes \widehat P_{\sigma} \ar[r] & 0 
\\ 0 \ar[r] & M' \otimes A\{\xi/\eta\} \ar[r] \ar@{^{(}->}[u] & M \otimes A\{\xi/\eta\} \ar[r] \ar@{^{(}->}[u] & M'' \otimes A\{\xi/\eta\} \ar[r] \ar@{^{(}->}[u] & 0
}
\]
Using proposition \ref{cricv}, the assertion results from an easy diagram chasing.
\end{proof}

\begin{rmks}
\begin{enumerate}
\item
As a particular case, we see that $A$ is $\eta$-convergent if and only if the twisted Taylor map of radius $\eta$
\begin{equation} \label{tildeta}
\xymatrix@R0cm{ \theta_{\eta} : & A \ar[r] & A\{\xi/\eta\} 
\\ & z \ar@{|->}[r] & \widetilde z := \sum_{k=0}^{\infty} \partial^{[k]}_{\sigma}(z)\xi^{(k)}}
\end{equation}
is well defined.
Note that if $A$ is the ring of functions on an affinoid domain of the $R$-line, and $x$ is a coordinate on the $R$-line, then there may exist at most one morphism of $R$-algebras $A \to A\{\xi/\eta\}$ sending $x$ to $\widetilde x$.
\item
If $A$ is $\eta$-convergent and we linearize the twisted Taylor map of radius $\eta$, then we obtain an $A$-linear map
\[
\widetilde \theta_{\eta} : P \to A\{\xi/\eta\}, \quad z_{1} \otimes z_{2} \mapsto z_{1}\theta_{\eta}(z_{2})
\]
and we have a commutative diagram
\begin{equation} \label{thetac}
\xymatrix{&& A \ar[ld]_{\theta} \ar[d]^{\theta_{\eta}}\ar[rd]^{\widehat \theta} \\ A[\xi] \ar@{^{(}->}[r] & P \ar[r] & A\{\xi/\eta\} \ar@{^{(}->}[r] & \widehat P_{\sigma}
}
\end{equation}
\item The twisted Taylor map of radius $\eta$ of $A$ is a morphism of affinoid $R$-algebras:
actually, all the maps in the diagram \eqref{thetac} are morphisms of rings (and even morphisms of $A$-algebras on the bottom line).
\end{enumerate}
\end{rmks}

\section{Twisted differential operators of finite level}

We let $(A, \sigma)$ be an $\eta$-convergent twisted affinoid $R$-algebra with respect to a twisted coordinate $x$ and a given contractive norm (this implies $\eta \geq \rho := \rho_{x}(\sigma)$).

\begin{dfn}
The $A$-module structure induced on $A\{\xi/\eta\}$ by the twisted Taylor map of radius $\eta$ will be called the \emph{right structure}.
\end{dfn}

Remember that, when we write $A\{\xi/\eta\} \otimes' -$, it means that we use the right structure on the left hand side:
\[
f \otimes' zs = \theta_{\eta}(z)f \otimes' s.
\]

\begin{lem} \label{uniq}
If $M, N$ are two finite $A$-modules, then the obvious map
$$
M \to A\{\xi/\eta\} \otimes_{A}'M, \quad s \mapsto 1 \otimes' s
$$
induces an injective $P$-linear map
\[
\mathrm{Hom}_{A\mathrm{-cont}}(A\{\xi/\eta\} \otimes_{A}'M, N) \hookrightarrow \mathrm{Hom}_{R\mathrm{-cont}}(M, N).
\]
\end{lem}

\begin{proof}
Recall from the commutative diagram \eqref{thetac} that the inclusion of $A[\xi]$ into $A\{\xi/\eta\}$ splits as
\[
A[\xi] \longrightarrow P \overset {\widetilde \theta_{\eta}}\longrightarrow A\{\xi/\eta\}
\]
It follows that the image of $\widetilde \theta_{\eta}$ is dense and therefore we have an injective map
\[
\mathrm{Hom}_{A\mathrm{-cont}}(A\{\xi/\eta\} \otimes_{A}'M, N) \hookrightarrow \mathrm{Hom}_{A\mathrm{-cont}}(P \otimes'_{A} M, N).
\]
Notice that the right hand side is nothing but $\mathrm{Hom}_{R\mathrm{-cont}}(M, N)$.
\end{proof}

\begin{dfn}
Assume that $M, N$ are two finite $A$-modules.
An $R$-linear map $\varphi : M \to N$ is called a \emph{twisted differential operator of level $\eta$} if it extends to a (necessarily unique) continuous $A$-linear map $\widetilde \varphi_{\eta} : A\{\xi/\eta\} \otimes_{A}'M \to N$, called its \emph{$\eta$-linearization}.
\end{dfn}

Note that uniqueness follows from lemma \ref{uniq}.
We will denote by $\mathrm{Diff}_{\sigma}^{(\eta)}(M, N)$ the set of all twisted differential operators of level $\eta$.
\begin{prop}
Assume that $M, N$ be two finite $A$-modules.
Then, $\mathrm{Diff}_{\sigma}^{(\eta)}(M, N)$ is a $P$-submodule of $\mathrm{Hom}_{R\mathrm{-cont}}(M, N)$ containing $\mathrm{Diff}_{\sigma}^{(\infty)}(M, N)$ and
\begin{equation} \label{duis}
\mathrm{Hom}_{A\mathrm{-cont}}(A\{\xi/\eta\} \otimes_{A}'M, N) \simeq \mathrm{Diff}_{\sigma}^{(\eta)}(M, N).
\end{equation}
\end{prop}

\begin{proof}
The last assertion is an immediate consequence of the definition (one uses lemma \ref{uniq} though), the first one results from the fact that restriction is $P$-linear and the middle one is obtained from the canonical surjections $A\{\xi/\eta\} \to P_{(n)}$. \end{proof}

\begin{rmks}
\begin{enumerate}
\item
The isomorphism \eqref{duis} turns $\mathrm{Diff}_{\sigma}^{(\eta)}(M, N)$ into a normed $A\{\xi/\eta\}$-module so that $\|\varphi\|_{\eta} = \|\widetilde \varphi_{\eta}\|$.
\item
Since $A\{\xi/\eta\}$ has an orthogonal Schauder basis and $M$ and $N$ are finite $A$-modules, one easily sees that $\mathrm{Diff}_{\sigma}^{(\eta)}(M, N)$ is a Banach $A$-module.
\item
If $\varphi : M \to N$ is a twisted differential operator of level $\eta$, then we have the following commutative diagrams
\[
\xymatrix{ M \ar[rrr]^-\varphi \ar@{^{(}->}[d]&&& N
\\ A\{\xi/\eta\} \otimes_{A}'M \ar[rrru]^-{\widetilde \varphi_{\eta}} \ar@{^{(}->}[d]&&&
\\ P \otimes_{A}'M \ar[rrruu]^-{\widetilde \varphi}, &&& 
}
\]
and
\[
\xymatrix@R0cm{A\{\xi/\eta\} \times \mathrm{Diff}_{\sigma}^{(\eta)}(M, N) \ar[rrr]\ar@{^{(}->}[dd] &&& \mathrm{Diff}_{\sigma}^{(\eta)}(M, N) \ar@{^{(}->}[dd]
\\ & (f, \varphi) \ar@{|->}[r] & f \cdot \varphi
\\ P \times \mathrm{Hom}_{R\mathrm{-cont}}(M, N) \ar[rrr] &&& \mathrm{Hom}_{R\mathrm{-cont}}(M, N),
}
\]
and we have
\[
\widetilde \varphi_{\eta}(f \otimes' s) = \widetilde \varphi(f \otimes' s) = (f \cdot \varphi) (s).
\]
\end{enumerate}
\end{rmks}

\begin{prop} \label{deldef}
There exists a unique map $\delta_{\eta}$ making commutative the following diagram:
\[
\xymatrix{ A \widehat \otimes_{R} A \ar@{=}[r] \ar[d]^{\delta}& \mathrm P \ar[r] \ar[d]^{\delta} & A\{\xi/\eta\}\ar[d]^{\delta_{\eta}} \ar@{->>}[r] & P_{(m+n)} \ar[d]^{\delta_{(m),(n)}}
\\ A \widehat \otimes_{R} A \widehat \otimes_{R} A \ar@{=}[r]& \mathrm P \widehat \otimes_{A}' \mathrm P \ar[r] & A\{\xi/\eta\} \widehat \otimes_{A}' A\{\xi/\eta\} \ar@{->>}[r] & P_{(m)} \otimes'_{A} P_{(n)}.
}
\]
Moreover, $\delta_{\eta}$ is a morphism of affinoid $R$-algebras of norm $1$.
\end{prop}

\begin{proof}
Uniqueness as well as being a morphism of rings follow immediately from the fact that the canonical map $A\{\xi/\eta\} \to \widehat P_{\sigma}$ is an injective morphism of rings.
Also, necessarily, one must have $\delta_{\eta}(\xi) = 1 \otimes' \xi + \xi \otimes' 1$.
Hence, existence and the assertion on the norm follow from the inequality $\|\delta(\xi)\|_{\eta} \leq \|\xi\|_{\eta} = \eta$ if we write $\| - \|_{\eta}$ also for the tensor product norm.
\end{proof}

\begin{dfn} The morphism $\delta_{\eta}$ is called the \emph{comultiplication map} on $A\{\xi/\eta\}$.
\end{dfn}

\begin{prop} \label{ring}
Twisted differential operators of level $\eta$ are stable under composition and we always have $\|\varphi \circ \psi\|_{\eta} \leq \|\varphi\|_{\eta}\|\psi\|_{\eta}$.
In particular, $\mathrm{Diff}_{\sigma}^{(\eta)}(M, M)$ is a sub-$R$-algebra of $\mathrm{End}_{R\mathrm{-cont}}(M)$ which is a Banach $R$-algebra.
\end{prop}

\begin{proof}
If $\varphi : M \to N$ and $\psi : L \to M$ are two twisted differential operators of level $\eta$, then there exists a commutative diagram
\begin{equation}\label{commut}
\xymatrix{P \otimes'_{A} L \ar[r]^-{\delta} \ar[d]& P \widehat \otimes'_{A} P  \otimes'_{A} L \ar[r]^-{\mathrm{Id} \otimes' \widetilde \psi} \ar[d]& P \otimes'_{A} M \ar[r]^-{ \widetilde \varphi} \ar[d]& N \ar[d]
\\ A\{\xi/\eta\} \otimes'_{A} L \ar[r]^-{\delta_{\eta}} & A\{\xi/\eta\} \widehat \otimes'_{A} A\{\xi/\eta\}  \otimes'_{A} L \ar[r]^-{\mathrm{Id} \otimes \widetilde \psi_{\eta}} & A\{\xi/\eta\} \otimes'_{A} M \ar[r]^-{ \widetilde \varphi_{\eta}} & N
}
\end{equation}
and we know that the upper line is exactly $\widetilde {\varphi \circ \psi}$.
Moreover, we have
\[
\|\varphi \circ \psi\|_{\eta} = \|\widetilde {(\varphi \circ \psi)}_{\eta}\| \leq \|\widetilde \varphi_{\eta} \| \|\mathrm{Id} \otimes \widetilde \psi_{\eta}\| \|\delta_{\eta}\|=\|\widetilde \varphi_{\eta} \| \|\widetilde \psi_{\eta}\| = \|\varphi\|_{\eta}\|\psi\|_{\eta}.
\qedhere
\]\end{proof}

\begin{prop} \label{teleta}
Let $M$ be a finite $A$-module and  $\theta_{\eta} : M \to M \otimes_{A} A\{\xi/\eta\}$ an $A$-linear map with respect to the right structure of $A\{\xi/\eta\}$.
Then $\theta_{\eta}$ is a twisted Taylor map of radius $\eta$ if and only if the diagram
\begin{equation} \label{twistay}
\xymatrix{M \ar[rr]^{\theta_{\eta}} \ar[d]^{\theta_{\eta}} && M \otimes_{A} A\{\xi/\eta\} \ar[d]^{\theta_{\eta} \otimes' \mathrm{Id}}
\\ M \otimes_{A} A\{\xi/\eta\} \ar[rr]^-{\mathrm{Id} \otimes \delta_{\eta}} && M \otimes_{A} A\{\xi/\eta\} \widehat \otimes'_{A} A\{\xi/\eta\}.}
\end{equation}
is commutative.
\end{prop}

\begin{proof}
We already know that the category of ($A$-finite) $\mathrm D_{\sigma}^{(\infty)}$-modules $M$ is equivalent to the category of finite $A$-modules endowed with a twisted Taylor structure.
And we saw in proposition \ref{cricv} that the twisted Taylor map will factor (uniquely) through $M \otimes_{A} A\{\xi/\eta\}$ if and only if $M$ is $\eta$-convergent.
Finally, proposition \ref{deldef} implies the commutativity of the diagram.
The converse also follows from the same proposition.
\end{proof}

\begin{rmks}
\begin{enumerate}
\item
One could define abstractly a twisted Taylor map of radius $\eta$ as a map $M \to M \otimes_{A} A\{\xi/\eta\}$ making commutative the diagram \eqref{twistay}.
Then, the proposition says that we would obtain a category which is equivalent to the category of $A$-finite $\mathrm D_{\sigma}^{(\infty)}$-modules that are $\eta$-convergent.
\item 
In the particular case $M=A$, the proposition implies that the following diagram is commutative:
\[
\xymatrix{A \ar[rr]^{\theta_{\eta}} \ar[d]^{\theta_{\eta}} && A\{\xi/\eta\} \ar[d]^{\theta_{\eta} \otimes' \mathrm{Id}} \\ A\{\xi/\eta\} \ar[rr]^-{\delta_{\eta}} && A\{\xi/\eta\} \widehat \otimes'_{A} A\{\xi/\eta\}.}
\]
\end{enumerate}
\end{rmks}

\section{Rings of twisted differential operators of finite level}

We let $(A, \sigma)$ be an $\eta$-convergent twisted affinoid $R$-algebra with respect to a twisted coordinate $x$ and a contractive norm.
We denote by $\partial^{[k]}_{\sigma}$ the standard twisted differential operator of order $k$ associated to $x$.

\begin{dfn}
The \emph{ring of twisted differential operators of level $\eta$} is $\mathrm D^{(\eta)}_{\sigma} := \mathrm{Diff}_{\sigma}^{(\eta)}(A, A)$.
\end{dfn}

Since, by proposition \ref{ring}, $\mathrm D^{(\eta)}_{\sigma}$ is the subring of $\mathrm{End}_{R\mathrm{-cont}}(A)$ made of all continuous $R$-linear endomorphisms $\varphi$ such that there exists a (necessarily unique) continuous $A$-linear form $\widetilde \varphi_{\eta} : A\{\xi/\eta\} \to A$ with $\varphi =\widetilde \varphi_{\eta} \circ \theta_{\eta}$, we get for free the analytic version of the density lemma 7.2 of \cite{LeStumQuiros18}:

\begin{lem}[Analytic density] \label{andens}
There exists a canonical isomorphism of Banach $A$-modules
\[
\mathrm D_{\sigma}^{(\eta)} \simeq \mathrm{Hom}_{A\mathrm{-cont}}(A\{\xi/\eta\}, A). \quad \Box
\]
\end{lem}

Recall also from the same proposition \ref{ring} that $\mathrm D^{(\eta)}_{\sigma}$ is actually a Banach algebra for the norm induced by the density lemma:
\[
\|\varphi\|_{\eta} := \|\widetilde \varphi_{\eta} \|:= \sup_{f\neq0} \|\widetilde \varphi_{\eta}(f)\|/\|f\|.
\]

\begin{xmp}
If we apply the analytic density lemma to $A = R\{x\}$ and $\eta = 1$, we see that
\[
\mathrm D_{\sigma}^{(1)} \simeq \mathrm{End}_{R\mathrm{-cont}}(A),
\]
which means that any continuous endomorphism of $R\{x\}$ is a bounded differential operator.
\end{xmp}

\begin{cor} \label{rigndu}
The $A$-module $\mathrm{Hom}_{A\mathrm{-cont}}(A\{\xi/\eta\}, A)$ is a (non commutative) Banach $R$-algebra for the multiplication defined by
\[
\xymatrix@R0cm{\Phi \Psi : & A\{\xi/\eta\} \ar[r]^-{\delta_{\eta}} & A\{\xi/\eta\} \widehat \otimes'_{A} A\{\xi/\eta\} \ar[r]^-{\mathrm{Id} \otimes' \Psi} & A\{\xi/\eta\} \ar[r]^-{ \Phi} & A. & \Box}
\]
\end{cor}

By duality, from the bottom line of diagram \eqref{thetac}, one obtains a series of inclusion maps
\begin{equation} \label{dcomp}
\xymatrix{
\mathrm D_{\sigma}^{(\infty)} \ar@{^{(}->}[r] & \mathrm D_{\sigma}^{(\eta)} \ar@{^{(}->}[r] \ar@{^{(}->}[rd] &  \widehat {\mathrm D}_{\sigma}^{(\infty)}
\\
&&\mathrm{End}_{R\mathrm{-cont}}(A), \ar[u]
}
\end{equation}
and the right upper map is explicitly given by
\[
\varphi \mapsto \sum_{k=0}^\infty \widetilde \varphi_{\eta}(\xi^{(k)}) \partial_{\sigma}^{[k]}.
\]

\begin{prop}
The inclusion map $\mathrm D_{\sigma}^{(\eta)} \hookrightarrow  \widehat {\mathrm D}_{\sigma}^{(\infty)}$ induces an isometric isomorphism of Banach $A$-modules
\[
\xymatrix@R0cm{
\mathrm D_{\sigma}^{(\eta)} \ar[r]^-\simeq & \left\{\sum_{k=0}^\infty z_{k} \partial_{\sigma}^{[k]}, \quad \exists C > 0, \forall k \in \mathbb N, \|z_{k}\| \leq C\eta^k\right\}
}
\]
for the sup norm
\[
\left\| \sum_{k=0}^\infty z_{k} \partial_{\sigma}^{[k]}\right\|_{\eta} = \sup \{\|z_{k}\|/\eta^k\}.
\]
\end{prop}

\begin{proof}
Let $\varphi \in \mathrm D_{\sigma}^{(\eta)}$.
Since $\widetilde \varphi_{\eta}$ is continuous, we have for all $k \in \mathbb N$,
\[
\|\widetilde \varphi_{\eta}(\xi^{(k)})\| \leq \|\widetilde \varphi_{\eta}\|\|\xi^{(k)}\| = \|\widetilde \varphi_{\eta}\|\eta^k
\]
and it follows that the canonical map $\mathrm D_{\sigma}^{(\eta)} \hookrightarrow  \widehat {\mathrm D}_{\sigma}^{(\infty)}$ takes its values inside the right hand side as asserted.
It also follows that the corresponding map has norm at most $1$.
Conversely, assume that we are given some $\sum_{k=0}^\infty w_{k} \partial_{\sigma}^{[k]}$ of norm (at most) $C$ in the right hand side.
Then, there exists a unique $\varphi \in \mathrm D_{\sigma}^{(\eta)}$ such that
\[
\widetilde \varphi_{\eta}(\sum_{k=0}^\infty z_{k} \xi^{(k)}) = \sum_{k=0}^\infty z_{k}w_{k} \in A
\]
because $\|z_{k}w_{k}\| \leq C \|z_{n}\|\eta^k \to 0$.
This defines an inverse map for our inclusion and the same inequality also implies that $\|\varphi\| \leq C$.
This shows that our map is an isometry.
\end{proof}

\begin{rmks}
\begin{enumerate}
\item
We will usually identify $\mathrm D_{\sigma}^{(\eta)}$ with its image inside $\widehat {\mathrm D}_{\sigma}^{(\infty)}$.
Note that we could have chosen to define it this way, but we do want an $A$-algebra and not merely an $A$-module.
\item
There exists a (left perfect) topological pairing of Banach $A$-modules
\begin{equation} \label{pairc}
A\{\xi/\eta\} \times  \mathrm D_{\sigma}^{(\eta)} \to  A, \quad (\xi^{(n)},\partial_{\sigma}^{[k]}) \mapsto \left\{\begin{array}{ccc}1 \ \mathrm{if} \ n = k \ \\ \ 0 \ \mathrm{otherwise} \end{array}\right..
\end{equation}
\item
The action of $A\{\xi/\eta\}$ on  $\mathrm D_{\sigma}^{(\eta)}$ is
given by
\begin{equation} \label{indfo}
\xi \cdot  \sum_{k=0}^\infty z_{k} \partial_{\sigma}^{[k]} =  \sum_{k=0}^\infty \left(z_{k+1} - z_{k}(x - \sigma^k(x)) \right) \partial_{\sigma}^{[k]},
\end{equation}
and one recovers the paring \eqref{pairc} by composition with evaluation at $1$.
\end{enumerate}
\end{rmks}

\begin{prop} \label{direct}
An $\eta$-convergent $A$-finite $\mathrm D_{\sigma}^{(\infty)}$-module $M$ extends canonically to a $\mathrm D_{\sigma}^{(\eta)}$-module.
\end{prop}

\begin{proof} \label{stdm}
The process is standard but requires some care.
We first consider the natural $A\{\xi/\eta\}$-linear map
\begin{equation}
\xymatrix@R=0cm{
M \otimes_{A} A\{\xi/\eta\} \ar[r] & \mathrm{Hom}_{A}(\mathrm{Hom}_{A}(A\{\xi/\eta\}, A), M) \\
s \otimes f \ar@{|->}[r] & (\Phi \mapsto \Phi(f)s)
}.
\end{equation}
Next, we extend linearly the twisted Taylor map of radius $\eta$ of $M$ in order to get an  $A\{\xi/\eta\}$-linear map $A\{\xi/\eta\} \otimes' M \to M \otimes A\{\xi/\eta\}$ and we compose on the left hand side.
Finally, we use restriction to continuous maps and the isomorphism $D_{\sigma}^{(\eta)} \simeq \mathrm{Hom}_{A\mathrm{-cont}}(A\{\xi/\eta\}, A)$ on the right hand side, and we obtain an $A\{\xi/\eta\}$-linear map
$$
A\{\xi/\eta\} \otimes_{A}' M \to \mathrm{Hom}_{A}(\mathrm D_{\sigma}^{(\eta)}, M),
$$
or equivalently, an $A\{\xi/\eta\}$-linear map
$$
\mathrm D_{\sigma}^{(\eta)} \to \mathrm{Hom}_{A}(A\{\xi/\eta\} \otimes' M, M) = \mathrm{Diff}^{(\eta)}_{\sigma}(M,M).
$$
Using the commutativity of \eqref{twistay}, one can check that this is a morphism of rings.
By construction, it is compatible with the natural map
$$
\mathrm D_{\sigma}^{(\infty)} \to \mathrm{Diff}^{(\infty)}_{\sigma}(M,M)
$$
that gives the action of $\mathrm D_{\sigma}^{(\infty)}$ on $M$.
\end{proof}

There is no exact converse to this proposition because the canonical map
$$
M \otimes_{A} A\{\xi/\eta\} \to \mathrm{Hom}_{A}(\mathrm D_{\sigma}^{(\eta)}, M)
$$
is far from being bijective in general (the right hand side is way too big as the case $M = A$ shows).
However, there exists a partial converse that will be used later:

\begin{prop} \label{prim}
If $M$ is an $A$-finite $\mathrm D^{(\eta)}_{\sigma}$-module then $M$ is $\eta\dagger$-convergent.
\end{prop}

\begin{proof}
We have to show that $M$ is $\eta'$-convergent whenever $\eta' < \eta$.

It follows from Banach open mapping theorem that $M$ is actually a topological $\mathrm D^{(\eta)}_{\sigma}$-module (see lemme 4.1.2 of \cite{Berthelot96} for example).
As a consequence, for all $s \in M$, the map
$$
\mathrm D^{(\eta)}_{\sigma} \to M, \quad \varphi \mapsto \varphi(s)
$$
will be a continuous $A$-linear map.
This means that there exists a constant $C$ such that for all $\varphi \in \mathrm D^{(\eta)}_{\sigma}$, we have $\|\varphi(s)\| \leq C \|\varphi\|\|s\|$.
Therefore, for all $k \in \mathbb N$, we have $\|\partial_{\sigma}^{[k]}(s)\| \leq \frac { C\|s\|}{\eta^k}$, and $\|\partial_{\sigma}^{[k]}(s)\|\eta'^k \leq C\|s\| (\frac {\eta'}\eta)^k \to 0$.
\end{proof}

\section{Deformation}

As before, we let $(A, \sigma)$ be a twisted affinoid $R$-algebra with twisted coordinate $x$.
We denote by $\rho$ the $x$-radius of $\sigma$ on $A$ with respect to $x$ and some contractive norm.
We also fix some $\eta \geq \rho$ such that $A$ is $\eta$-convergent 

\begin{prop} \label{indeta}
Assume that $x$ is also a twisted coordinate for some other endomorphism $\tau$ of $A$ and that $\eta \geq \rho(\tau)$.
Then, $A$ is also $\eta$-convergent with respect to $\tau$ with the \emph{same} twisted Taylor morphism as $\sigma$.
\end{prop}

\begin{proof}
We may consider the following commutative diagram
\[
\xymatrix{&& A \ar[ld]_{\theta} \ar[d]^{\theta_{\eta}}
\\ A[\xi] \ar@{^{(}->}[r] \ar@{->>}[d] & P \ar[r] \ar@{->>}[d] & A\{\xi/\eta\} \ar@{->>}[d]
\\ A[\xi]/\xi^{(n)_{\tau}} \ar@{=}[r] \ar@/_2pc/[rr] & P_{(n)_{\tau}}  & A\{\xi/\eta\}/\xi^{(n)_{\tau}}
}
\]
\\
(where $\theta_{\eta}$ denotes the twisted Taylor map of radius $\eta$ with respect to $\sigma$).
A quick look shows that the composition of the right vertical maps does not depend on $\sigma$.
Moreover, since we assume that $\rho(\tau) \leq \eta$, then the curved arrow is an isomorphism by proposition \ref{comps}, and going to the limit on $n$ provides us with a commutative diagram
\[
\xymatrix{& A \ar[ld]_{\theta} \ar[d]^{\theta_{\eta}}
\\ P \ar[d] & A\{\xi/\eta\} \ar@{^{(}->}[ld]
\\ \widehat P_{\tau}.
}
\]
It then follows from proposition \ref{cricv} that $A$ is $\eta$-convergent with respect to $\tau$ and that $\theta_{\eta}$ is the twisted Taylor map of radius $\eta$ with respect to $\tau$. \end{proof}

\begin{cor} \label{indsig}
The ring structure of $\mathrm{Hom}_{A\mathrm{-cont}}(A\{\xi/\eta\}, A)$ does \emph{not} depend on $\sigma$.
\end{cor}

\begin{proof}
Follows from the definition in corollary \ref{rigndu} of the multiplication rule on the $A$-module $\mathrm{Hom}_{A\mathrm{-cont}}(A\{\xi/\eta\}, A)$ and proposition \ref{indeta}.
\end{proof}

\begin{thm}[Deformation] \label{rnghom}
Let $\tau$ be another $R$-endomorphism of $A$ such that $x$ is also a $\tau$-coordinate.
If $\eta \geq \rho(\tau)$, then there exists an isometric $A$-linear isomorphism of $R$-algebras
\[
\xymatrix@R0cm{ \mathrm D_{A/R,\sigma}^{(\eta)} \ar[r]^-{\simeq} & \mathrm D_{A/R,\tau}^{(\eta)}}.
\]
that only depend on $x$.
\end{thm}

\begin{proof}
This follows from corollary \ref{indsig} since the analytic density isomorphisms provided by lemma \ref{andens} are $A$-linear isomorphisms of $R$-algebras.
\end{proof}

As a consequence of the theorem, we can (and will) \emph{identify} two twisted differential operators of level $\eta$ with respect to two different endomorphisms when they have the same linearization.

\begin{rmk}
This isomorphism is explicit in the sense that, as in proposition 7.6 of \cite{LeStumQuiros18}, we have for example
\[
\partial_{\sigma} = \sum_{k=1}^\infty \left( \prod_{i=1}^{k-1} \left( \sigma(x) - \tau^i(x)\right)\right) \partial_{\tau}^{[k]}.
\]
\end{rmk}

The particular case $\tau = \mathrm{Id}_{A}$ is worth stating separately:

\begin{cor} \label{lastco}
If $A$ is formally smooth (of relative dimension one) over $R$ and $x$ is also an \'etale coordinate on $A$, then there exists an isometric $A$-linear isomorphism of $R$-algebras
\[
\xymatrix@R0cm{ \mathrm D_{\sigma}^{(\eta)} \ar[r]^-{\simeq} & \mathrm D^{(\eta)}} 
\]
that only depends on $x$. 
\end{cor}

This is the \emph{analytic deformation map}.

\begin{rmk}
Again, this isomorphism is explicit.
For example, as shown in corollary 7.7 of \cite{LeStumQuiros18}, we will have
\[
\partial_{\sigma} = \sum_{k=1}^\infty ( \sigma(x) - x)^{k-1} \partial^{[k]} \quad \mathrm{and} \quad \sigma = \sum_{k=0}^\infty ( \sigma(x) - x)^{k} \partial^{[k]}.
\]
\end{rmk}

\section{Confluence}

We still let $(A, \sigma)$ be a twisted affinoid $R$-algebra with twisted coordinate $x$.
We denote by $\rho$ the $x$-radius of $\sigma$ on $A$ with respect to $x$ and some contractive norm.
We also fix some $\eta > \rho$ such that $A$ is $\eta\dagger$-convergent

\begin{dfn}
The ring of \emph{twisted differential operators of level $\eta\dagger$} is
$$
\mathrm D^{(\eta\dagger)}_{\sigma} := \varinjlim_{\eta'<\eta} \mathrm D^{(\eta')}_{\sigma}.
$$
\end{dfn}

\begin{rmks}
\begin{enumerate}
\item
We have the following description:
\begin{align*}
\mathrm D_{\sigma}^{(\eta\dagger)} & = \left\{\sum_{k=0}^\infty z_{k} \partial_{\sigma}^{[k]}, \quad \exists \eta' < \eta, \exists C > 0, \forall k \in \mathbb N, \|z_{k}\| \leq C\eta'^k\right\}
\\
& = \left\{\sum_{k=0}^\infty z_{k} \partial_{\sigma}^{[k]}, \quad \exists \eta' < \eta, \frac {\|z_{k}\|}{\eta'^k}\to 0\right\}.
\end{align*}
\item
When $\sigma = \mathrm{Id}_{A}$ (so that we drop $\sigma$ from our notations), this definition is closely related to the theory of arithmetic differential operators of Pierre Berthelot (see proposition 2.4.4 of \cite{Berthelot96}).
More precisely, if we let $\mathcal X$ be a (one dimensional) smooth affine formal $\mathcal V$-scheme with \'etale coordinate $x$ and if we set $A := \Gamma(\mathcal X, \mathcal O_{\mathcal X\mathbb Q})$, then we have
\[
\Gamma(\mathcal X, \mathcal D_{\mathcal X\mathbb Q}^{\dagger}) = \mathrm D_{A/K}^{(1\dagger)}.
\]
\item
More generally, still assuming $\sigma = \mathrm{Id}_{A}$, let $\mathcal X_{0}$ be a smooth formal $\mathcal V$-scheme with \'etale coordinate $x$, $\mathcal X$ a blowing up of $\mathcal X_{0}$ centered in an ideal containing $\pi^k$ ($\pi$ a uniformizer) and $\mathcal Y = \mathrm{Spf}(A)$ an affine open subset of $\mathcal X$.
Then, with the notations of \cite{HuygheSchmidtStrauch17*}, we have
\[
\Gamma(\mathcal Y, \mathcal D_{\mathcal X,k}^{\dagger}) = \mathrm D_{A/K}^{(\eta\dagger)}
\]
with $\eta = |\pi|^k$.
\end{enumerate}
\end{rmks}

\begin{prop} \label{catet}
The category of $A$-finite $\mathrm D^{(\eta\dagger)}_{\sigma}$-modules is equivalent (isomorphic) to the subcategory of all $\eta\dagger$-convergent $A$-finite $\mathrm D^{(\infty)}_{\sigma}$-modules.
\end{prop}

\begin{proof}

Follows from proposition \ref{prim} and \ref{direct}.
\end{proof}

\begin{thm}[Confluence]
Let $\tau$ be another $R$-endomorphism of $A$ such that $x$ is also a $\tau$-coordinate.
If $\eta > \rho(\tau)$, then the categories of $\eta\dagger$-convergent $A$-finite $\mathrm D^{(\infty)}_{\sigma}$-modules with respect to $\sigma$ and $\tau$ are canonically equivalent.
\end{thm}

\begin{proof}
Follows from the deformation theorem \ref{rnghom} and proposition \ref{catet}.
\end{proof}

\begin{dfn}
Assume that $x$ is a $q$-coordinate on $A$ and that all positive $q$-integers are invertible in $R$.
Then, a $\nabla_{\sigma}$-module $M$ on $A$ is said to be \emph{$\eta'$-convergent} if
\[
\forall s \in M, \quad \left\|\frac{\nabla^{k}_{\sigma}(s)}{(k)_{q}!}\right\|\eta'^k \to 0\ \mathrm{when}\ k \to \infty,
\]
and $\eta\dagger$-convergent if it is $\eta'$-convergent whenever $\eta' < \eta$.
\end{dfn}

We will denote by $\nabla_{\sigma}\mathrm{-Mod}^{(\eta\dagger)}(A/R)$ the full subcategory of all finite $\nabla_{\sigma}$-modules $M$ that are $\eta\dagger$-convergent.

The condition that positive $q$-integers are invertible is necessary for the definition to make sense.
If $q \in K$, then this happens exactly when $q$ is not a root of unity or else if $q = 1$ and $\mathrm{Char}(K) = 0$.

\begin{prop} \label{invis}
Assume that $x$ is a $q$-coordinate and that all positive $q$-integers are invertible.
If $A$ is $\eta\dagger$-convergent, then there exists an equivalence of categories
\[
\nabla_{\sigma}\mathrm{-Mod}^{(\eta\dagger)}(A/R) \simeq A\mathrm{-finite}\ \mathrm D^{(\eta\dagger)}_{A/R,\sigma}\mathrm{-Mod}
\]
which is compatible with cohomology in the sense that
\[
\mathrm{H}_{\partial_{\sigma}}^*(M) = \mathrm{Ext}^*_{D^{(\eta\dagger)}_{\sigma}}(A, M).
\]
\end{prop}

\begin{proof}
We know that, if $D_{\sigma}$ denotes the twisted Weyl algebra of $(A,\sigma)$, then the category of finite $\nabla_{\sigma}$-modules is equivalent to the category of finite $D_{\sigma}$-modules.
Moreover in our situation, we have $\mathrm D_{\sigma} = \mathrm D_{\sigma}^{(\infty)}$.
Therefore, since for all $k \in \mathbb N$, we have $\partial_{\sigma}^k = (k)_{q}!\partial_{\sigma}^{[k]}$, we can identify $\eta$-convergent $\nabla_{\sigma}$-modules and $\eta$-convergent $\mathrm D^{(\infty)}_{\sigma}$-modules.
Hence, the first result follows from proposition \ref{catet}.
In order to prove the second one, we may use the Spencer resolution
$$
[D^{(\eta\dagger)}_{\sigma} \overset {\partial_{\sigma}} \longrightarrow D^{(\eta\dagger)}_{\sigma}] \simeq A
$$
(the map is multiplication on the right by $\partial_{\sigma})$.
\end{proof}

As a consequence of the confluence theorem, we obtain the following theorem (which is essentially theorem 6.3 ii) of \cite{Pulita08}) that we state with all the necessary hypothesis for the convenience of the reader:

\begin{thm}[Confluence-bis]\label{confl}
Let $K$ be a non trivial complete ultrametric field of characteristic zero.
Let $R$ be an affinoid $K$-algebra and $q \in R$ such that all positive $q$-integers are invertible.
Let $A$ be a formally smooth affinoid $R$-algebra (of relative dimension one) such that $A$ is $\eta\dagger$-convergent with respect to some \'etale coordinate $x$ for some $\eta \geq 0$.
Assume that $x$ is a $q$-coordinate for an $R$-algebra endomorphism $\sigma$ of $A$ with $\eta \geq \rho_{x}(\sigma)$.

Then there exists an equivalence of categories
\[
\nabla_{\sigma}\mathrm{-Mod}^{(\eta\dagger)}(A/R) \simeq \nabla\mathrm{-Mod}^{(\eta\dagger)}(A/R)
\]
which is compatible with cohomology on both sides.
\end{thm}

\begin{proof}
This follows from proposition \ref{invis} and corollary \ref{lastco}.
\end{proof}

\begin{cor}
In the situation of the theorem, if we assume that $x$ is a strong quantum coordinate for $\sigma$, then there exists a fully faithful functor
\begin{equation} \label{last}
\nabla\mathrm{-Mod}^{(\eta\dagger)}(A/R) \hookrightarrow \sigma_{A}\mathrm{-Mod}(A/R)
\end{equation}
which is compatible with cohomology on both sides. $\quad \Box$
\end{cor}

We may consider this functor as an embedding and identify an $\eta\dagger$-convergent $\nabla$-module with the corresponding quantum module.

\begin{rmks}
\begin{enumerate}
\item The hypothesis on $K$ are satisfied when $K = \mathbb Q_{p}$ or $K = \mathbb C_{p}$ or $K = \mathbb Q((T))$ for example.
\item The hypothesis on $q$ is satisfied for example if $q \in K$ and $q$ is not a root of $1$.
\item The extra condition in the corollary is satisfied if $x \in A^\times$ and $1-q \in R^\times$.
\item  When $R=K$ and $A = K\{x/r, r_{1}/x\}$, the conditions on $A$ (and $\sigma$) can be summarized (when $q$ is not a root of $1$) as
\[
|h| \leq \eta < r_{1} \quad \mathrm{and} \quad |1 - q| \leq \frac \eta r.
\]
\item
The functor \eqref{last} is completely explicit: if we are given a finite $A$-module $M$, endowed an $\eta\dagger$-convergent connection, then the corresponding twisted module will be given by
\begin{equation} \label{expli}
\sigma_{M}(s) = \sum_{k=0}^\infty  \frac {((q-1)x+h)^{k}}{k!} \partial_{M}^{k}(s).
\end{equation}
In other words, it is obtained by replacing $\xi$ by $\sigma(x) -x$ in the twisted Taylor series:
\[
\sigma_{M}(s) = \theta_{\eta}(s)(\sigma(x) -x).
\]
Formally, we may also write $\sigma_{M} = \exp((\sigma(x)-x)\partial_{M})$.
\item In the case $h =0$, one can also use the logarithmic derivative:
if we set
\[
q^{x\partial} := \sum_{k=0}^\infty \log(q)^k(x\partial)^k,
\]
then we have for all $n \in \mathbb N$,
\[
q^{x\partial}(x^n)= \sum_{k=0}^\infty \log(q)^k(x\partial)^k(x^n) = \sum_{k=0}^\infty \log(q)^kn^kx^n = q^nx^n = \sigma(x^n)
\]
and it formally follows that
\[
\sigma_{M}(s) = (q^{x\partial_{M}})(s) \ \left(= \sum_{k=0}^\infty \log(q)^k(x\partial_{M})^k(s)\right).
\]
\item
Compatibility with cohomology means that $H^*_{\partial}(M) = H^*_{\sigma}(M)$.
In particular, the solutions of the differential system associated to the $\nabla$-module $M$ are exactly the same as the solutions of the functional system associated to the $\sigma$-module $M$.
\end{enumerate}
\end{rmks}

\begin{xmps}
\begin{enumerate}
\item
If we consider the differential equation $\partial (s) = cs$, then we will have $\partial^k(s) = c^ks$ for all $k \in \mathbb N$ and $\sigma$ will be multiplication by 
\[
\sum_{k=0}^\infty  \frac {c^k((q-1)x+h)^{k}}{k!} = \exp(c((q-1)x+h)).
\]
Of course, the differential equation $\partial (s) = cs$ has the same solution $\exp(cx)$ as the functional equation
\[
s(qx+h) = \exp(c((q-1)x+h))s(x).
\]
\item
If we consider now the differential equation $\partial (s) = \frac a xs$, then we will have
\[
\partial^k(s) = \frac {a(a-1) \cdots (a-k+1)} {x^k}s
\]
for all $k \in \mathbb N$ and $\sigma$ will be the multiplication by
\[
\sum_{k=0}^\infty  \frac {a(a-1) \cdots (a-k+1)} {k!}  \left(q + \frac hx - 1\right)^{k} = \left(q + \frac hx\right)^a.
\]
Again, one easily checks that the differential equation $\partial (s) = \frac a xs$ has the same solution $x^a$ as the functional equation
\[
s(qx+h) = \left(q + \frac hx\right)^as(x).
\]
\end{enumerate}
\end{xmps}

\addcontentsline{toc}{section}{References}
\printbibliography

\end{document}